\documentclass[12pt]{amsart}

\usepackage{amssymb}
\usepackage{amsmath}
\usepackage{graphicx} 
\usepackage{enumerate}
\usepackage{color}
\usepackage[pdftex]{hyperref}
\usepackage[all]{xy}

\newcommand{\C}{\mathbb C}

\newcommand{\R}{\mathbb R}

\newcommand{\Z}{\mathbb Z}

\renewcommand{\P}{\mathbb P}
\newcommand{\cA}{\mathcal A}

\newcommand{\cC}{\mathcal C}

\newcommand{\cI}{\mathcal I}
\newcommand{\cJ}{\mathcal J}

\newcommand{\fl}{\mathfrak l}
\newcommand{\fp}{\mathfrak p}

\newcommand{\bF}{\mathbf F}

\newcommand{\na}{\nabla}
\newcommand{\pa}{\partial}

\renewcommand{\a}{\alpha} 
\renewcommand{\b}{\beta} 

\newcommand{\G}{\varGamma}
\newcommand{\De}{\mathit{\Delta}}

\newcommand{\e}{\varepsilon}
\newcommand{\f}{\varphi}
\renewcommand{\l}{\lambda}

\newcommand{\w}{\omega}
\newcommand{\W}{\mathit{\Omega}}
\newcommand{\la}{\langle}
\newcommand{\ra}{\rangle}
\newcommand{\dd}{\mathrm{d}}
\newcommand{\dt}{\mathrm{d}_t}

\newcommand{\tx}{\widetilde{x}}
\newcommand{\tr}{\;^t}

\newcommand{\bu}{\bullet}
\newcommand{\diag}{\mathrm{diag}}

\newcommand{\TP}[1]{
  \sideset{^t}{}{\mathop{#1}}
}
\newcommand{\under}[3]{
  {}_{#1} #2 _{#3}
}

\newcommand{\si}{\sigma}
\newcommand{\we}{\wedge}
\newcommand{\tpi}{2\pi \sqrt{-1}}
\newcommand{\bS}{\mathbf{S}}
\newcommand{\ot}{\otimes}

\newcommand{\conti}[1]{U_{#1}}
\newcommand{\Conti}[1]{\mathcal{U}_{#1}}
\newcommand{\cohom}{V}
\newcommand{\cohomi}{{V^{(l)}}}
\newcommand{\Omd}{\W^{\bu}}
\newcommand{\van}{{\rm van}}
\newcommand{\bC}{b\mathcal{C}}
\newcommand{\nbc}{\mathbf{nbc}}
\newcommand{\fZ}{\mathfrak{Z}}

\newtheorem{theorem}{Theorem}[section]
\newtheorem{proposition}[theorem]{Proposition}
\newtheorem{lemma}[theorem]{Lemma}
\newtheorem{cor}[theorem]{Corollary}
\newtheorem{fact}[theorem]{Fact}
\theoremstyle{definition}

\newtheorem{exam}[theorem]{Example}
\newtheorem{definition}[theorem]{Definition}

\newtheorem{notation}[theorem]{Notation}

\makeatletter
 
 \@addtoreset{equation}{section}
\makeatother   
\title
[Intersection numbers]
{Intersection numbers of twisted cycles and cocycles for degenerate arrangements}

\author[Y. Goto]{Yoshiaki Goto}
\address[Goto]{
  General Education,
  Otaru University of Commerce,
  Otaru, Hokkaido 047-8501, Japan 
}
\email{goto@res.otaru-uc.ac.jp}

\date{\today}

\begin{document}
\maketitle
\begin{abstract}
  We study the intersection numbers defined on 
  twisted homology or cohomology groups
  that are associated with hypergeometric integrals corresponding to 
  degenerate hyperplane arrangements 
  in the projective $k$-space. 
  We present formulas to evaluate the intersection numbers 
  in the case when exactly one $(k+1)$-tuple of the hyperplanes 
  intersects at a point. 
  As an application, we discuss the contiguity relations 
  of hypergeometric functions 
  in terms of the intersection numbers on twisted cohomology groups. 
\end{abstract}

%\tableofcontents
\section{Introduction}
As a generalization of Gauss' hypergeometric function, 
we consider integrals of the form
\begin{align*}
  \int_{\square} \prod_{j=1}^{k+n+1} 
  (z_{0j}+z_{1j}t_1+\cdots+z_{kj}t_k)^{a_j} \ \dd t_1 \we \cdots \we \dd t_k ,
\end{align*}
where the $a_i$ are parameters, $z_{ij}$ are variables, and 
$\square$ denotes a certain region.
These \textit{hypergeometric integrals} 
can be studied in the framework of twisted homology and cohomology groups 
(see, e.g., \cite{AK}). 
The intersection forms on these homology or cohomology groups are
not only of theoretical interest, but can also be useful in deriving explicit formulas between 
hypergeometric integrals. 
In this paper, we focus on evaluating the intersection numbers for 
twisted (co)cycles, which are frequently used. 

The linear form $z_{0j}+z_{1j}t_1+\cdots+z_{kj}t_k$ in $t_1 ,\dots ,t_k$ 
(here, we regard the $z_{ij}$ as constants) 
defines an arrangement $\cA$ of hyperplanes in $\C^k$. 
Let $\cA^h$ be the arrangement of hyperplanes in $\P^k$ that consists of 
the homogenization of $\cA$ and the hyperplane at infinity. 
$\cA^h$ corresponds to a $(k+1)\times (k+n+2)$ matrix $z$ (see Section \ref{sec:preliminaries}).  
We say that $z$ is the ``coefficient matrix'' 
of $\cA^h$ (or $\cA$). 
In the case when $\cA^h$ is in a general position 
(i.e., any $(k+1)$-submatrix of $z$ is invertible), 
explicit formulas for the intersection numbers of
twisted cycles (resp. cocycles) defined by real chambers (resp. logarithmic $k$-forms) 
were derived in \cite{KY} (resp. \cite{CM}, \cite{M1}). 
Let $Z^{(0)}$ be the set of coefficient matrices 
that give arrangements of hyperplanes in a general position. 

The 
hypergeometric integrals associated with degenerate arrangements 
are also interesting; for example, 
the generalized hypergeometric function ${}_p F_{p-1}$ $(p\geq 3)$ and 
Lauricella's hypergeometric function $F_A$ are degenerate cases. 
In such examples, we evaluate the intersection numbers by blowing-up $\P^k$. 
Though the resulting evaluations have been derived 
by several authors, 
they are slightly complicated. 
For example, when we evaluate the intersection numbers for chambers, 
the orientations of the chambers in the blown-up space become complicated 
(see \cite{KY}). 
To investigate the orientations correctly, 
it seems that some geometric analysis is necessary. 

In this paper, 
we present formulas for evaluating the intersection numbers 
in the case when exactly one $(k+1)$-tuple of the hyperplanes 
intersects at a point. 
In other words, exactly one of the $(k+1)$-submatrices of 
the coefficient matrix $z_0$ 
is not invertible. 
We denote the set of such matrices by $Z^{(1)}$. 
Our formulas do not involve the blowing-up process, and can evaluate 
intersection numbers using those for $z\in Z^{(0)}$. 

The general idea of our method is as follows
(though we explain in terms of cohomology, 
a similar discussion is valid for the homology case).
Let $H^k(\W^\bu(T_{z}),\na^{\pm \a})$ and 
$H^k(\W^\bu(T_{z_0}),\na_0^{\pm \a})$ be
twisted cohomology groups corresponding to 
$z\in Z^{(0)}$
and $z_0\in Z^{(1)}$, respectively. 
We denote the intersection form on $H^k(\W^\bu(T_{z}),\na^{\pm \a})$ 
(resp. $H^k(\W^\bu(T_{z_0}),\na_0^{\pm \a})$) by
$\cI^c$ (resp. $\cI_0^c$). 
% We give a formula to evaluate the intersection form $\cI_0^c$ 
% on $H^k(\W^\bu(T_{z_0}),\na_0^{\pm \a})$ from 
% the intersection form $\cI^c$ on $H^k(\W^\bu(T_{z}),\na^{\pm \a})$. 
We are interested in $\cI_0^c (\f ,\psi)$, where
$\f \in H^k(\W^\bu(T_{z_0}),\na_0^{\a})$ and 
$\psi \in H^k(\W^\bu(T_{z_0}),\na_0^{-\a})$ are 
expressed as ``limits'' of the logarithmic forms 
$\f' \in H^k(\W^\bu(T_{z}),\na^{\a})$, 
$\psi' \in H^k(\W^\bu(T_{z}),\na^{-\a})$ as $z\to z_0$, respectively.
\begin{itemize}
\item There is one ``vanishing cocycle'' $\f^{\van}\in H^k(\W^\bu(T_{z}),\na^{\pm \a})$ 
  that becomes zero as $z\to z_0$. 
\item There are isomorphisms between $H^k(\W^\bu(T_{z_0}),\na_0^{\pm \a})$ 
  and the orthogonal complements of $\f^{\van}$ 
  in $H^k(\W^\bu(T_{z}),\na^{\pm \a})$ with respect to $\cI^c$, 
  and these preserve the intersection form. 
\item $\cI_0^c (\f ,\psi)$ is evaluated as the intersection number $\cI^c$ of  
  the projections of $\f', \psi'$ to these orthogonal complements. 
\end{itemize}
As the projections are also expressed by $\cI^c$, 
we can evaluate $\cI_0^c (\f ,\psi)$ using only $\cI^c$. 
The resulting formula does not require the blow-up process, 
which is one advantage of our method.  

We can consider more degenerated cases in which, for example, 
two submatrices of the coefficient matrix become zero. 
In such cases, we also consider 
the orthogonal complement of the space spanned by 
the vanishing (co)cycles, and the projection to them. 
In some examples, we can check that our discussions are valid,  
and hence the intersection numbers can be calculated easily. 
However, we do not have a proof for general situations 
(thus, our method does not yet give clear formulas for 
${}_p F_{p-1}$ $(p\geq 3)$ or $F_A$). 
The difficulty comes from the problem 
of moving a generic $z$ to a given degenerate $z_0$, 
which is also of interesting in terms of moduli theory. 

The remainder of this paper is arranged as follows. 
In Section \ref{sec:preliminaries}, we briefly review 
twisted homology and cohomology groups and intersection forms. 
Section \ref{section:vanishing} is the main part of this paper, 
in which we derive formulas for the intersection numbers. 
In Section \ref{section:examples}, we present several examples. 
In the first example, we evaluate the intersection numbers 
using our formulas 
for a degenerated case of $(3,6)$-type. 
In the second example, we show that using 
the intersection matrix, which is also evaluated by our formulas, allows certain sets of logarithmic forms 
to give bases of twisted cohomology groups.  
These bases are then used 
to investigate the contiguity relations
in Section \ref{section:contiguity}. 
As discussed in \cite{TGKT}, explicit expressions for the contiguity relations are 
useful in algebraic statistics. 
In \cite{GM-PC}, we presented such expressions for the case of a general position 
in terms of the intersection numbers of twisted cohomology groups. 
We extend this formulation to the ``degenerated version'' in Section \ref{section:contiguity}. 
The discussion is parallel to that in \cite{GM-PC}. 
Our formulas for intersection numbers 
can be used to obtain explicit expressions for the contiguity relations.

\section{Preliminaries}
\label{sec:preliminaries}
In this section, we review twisted homology and cohomology groups 
associated with hypergeometric integrals. 
% function defined by 
% an integral of product of complex powers of linear forms. 
For definitions and basic facts, refer to \cite[Chapter 2]{AK}. 
For twisted cohomology groups, we use the same notation as in \cite{GM-PC}.

\subsection{Settings}
Let $Z=Z_{k+1,k+n+2}$ be the set of $(k+1)\times (k+n+2)$ matrices of the form
\begin{align*}
  z=(z_{ij})_{
  \begin{subarray}{l}
    0\leq i \leq k \\ 0 \leq j \leq k+n+1
  \end{subarray}}
  =\bordermatrix{
     &0&1&2&\cdots& k+n+1\cr
    0&1&z_{01}&z_{02}&\cdots&z_{0,k+n+1} \cr
    1&0&z_{11}&z_{12}&\cdots&z_{1,k+n+1} \cr
    \vdots&\vdots &\vdots &\vdots &\ddots&\vdots \cr
    k&0&z_{k1}&z_{k2}&\cdots&z_{k,k+n+1} \cr
  } ,\quad z_{ij} \in \C .
\end{align*}
We regard the $0$-th column as $\tr (z_{00},z_{10},\dots ,z_{k0})=\tr (1,0,\dots,0)$. 
Let $\cJ$ be the set consisting of the subsets of 
$\{0,1,2,\dots,k+n,k+n+1\}$ with cardinality $k+1$. 
When we write $J=\{j_0 ,j_1 ,\dots ,j_k\} \in \cJ$, we simply consider $J$ as a set; 
however, when we write $J=\la j_0 j_1 \dots j_k\ra$, 
we regard $J$ as a set with an order. 
For example, 
\begin{align*}
  \{ 0,1,2 \} =\{ 0,2,1 \} ,\ \textrm{but} \ 
  \la 012 \ra \neq \la 021 \ra .
\end{align*}
% Any element in $\cJ$ is expressed as 
% \begin{align*}
%   J=\{j_0,j_1,\dots,j_k\}, \quad 
%   0\leq j_0<j_1<\cdots <j_k\leq k+n+1.
% \end{align*}
For $z\in Z$ and $J\in \cJ$, we set 
\begin{align*}
  z \la J\ra = z\la j_0 j_1 \dots j_k\ra =
  \begin{pmatrix}
    z_{0,j_0} &z_{0,j_1} &\cdots &z_{0,j_k} \\
    z_{1,j_0} &z_{1,j_1} &\cdots &z_{1,j_k} \\
    \vdots & \vdots &\ddots & \vdots \\
    z_{k,j_0} &z_{k,j_1} &\cdots &z_{k,j_k}
  \end{pmatrix},
\end{align*}
which is the sub-matrix of $z$ 
consisting of the $j_0$-th, $j_1$-th, $\dots$, $j_k$-th columns.
We consider two subsets of $Z$ defined by 
\begin{align*}
  Z^{(0)} &= \{ z\in Z \mid |z \la J\ra |\neq 0 \ (\forall J \in \cJ) \} , \\
  Z^{(1)} &= \left\{ z\in Z \left| 
      \begin{array}{l}
        \textrm{there exists }J\in \cJ 
        \textrm{ such that } \\
        |z \la J\ra |=0 \textrm{ and } |z \la J' \ra |\neq 0 \ (J' \neq J) 
      \end{array}
    \right. \right\} .
\end{align*}
In this paper, we consider matrices $z$ belonging to 
$Z^{(01)} =Z^{(0)} \cup Z^{(1)}$. 

Let $L_j=L_j (t)=L_j (t;z)$ $(0\le j\le k+n+1)$ 
be linear forms of $t_0,t_1,\dots,t_k$ defined by
\begin{align*}
  (t_0,t_1,\dots,t_k)z =(L_0,L_1,\dots,L_{k+n+1}).
\end{align*}
For example, 
\begin{align*}
  L_0 (t;z)=t_0 ,\quad 
  L_1 (t;z)=z_{01}t_0+z_{11}t_1+\cdots+z_{k1}t_k.
\end{align*}
We regard $(t_0,t_1,\dots,t_k)$ as the projective coordinates of $\P^k$ and 
$(t_1,\dots,t_k)$ as the affine coordinates by setting $t_0=1$.
We denote $\{ t\in \P^k \mid L_j (t;z)=0 \}$ by $(L_j=0)$. 
Thus, the collection
\begin{align*}
  \cA_z^h=\{ (L_j =0) \mid j=0,1,\dots , k+n+1 \} 
\end{align*}
defines an arrangement of hyperplanes in $\P^k$. 
We regard $(L_0 =0)$ as the hyperplane at infinity in $\P^k$. 
If $z\in Z^{(1)}$ with $|z \la J\ra |=0$, then 
$k+1$ hyperplanes $(L_j =0)$ ($j\in J$) intersect at a point. 
If $z\in Z^{(0)}$, then the hyperplanes in $\cA_z^h$ are in a general position
in $\P^k$. 
Note that we can regard 
\begin{align*}
  \cA_z=\{ (L_j =0) \mid j=1,\dots , k+n+1 \} 
\end{align*}
as an arrangement of hyperplanes in $\C^k =\P^k -(L_0=0)$. 

We set
\begin{align*}
  \fZ = \left\{ (t,z)\in \P^k\times Z^{(01)}\mid 
  \prod_{j=0}^{k+n+1} L_{j}(t;z) \ne 0 \right\} ,
\end{align*}
and denote $T_z$ as the preimage of $z$ under the projection 
$\fZ \ni (t,z) \mapsto z \in Z^{(01)}$. 
We identify $T_z$ with an open subset of $\P^k$ or $\C^k$, that is, 
\begin{align*}
  T_z = M(\cA_z) =\C^k -\bigcup_{k=1}^{k+n+1} (L_k=0) \subset \C^k \subset \P^k .
\end{align*}

To consider hypergeometric integrals, 
we use the complex parameters
$\a_0,\a_1,\dots,\a_{k+n},\a_{k+n+1}$
% in $\C-\Z$ 
that satisfy 
\begin{align}
  \label{eq:non-int}
  \a_0,\a_1,\dots,\a_{k+n},\a_{k+n+1} \not\in \Z ,\quad 
  \sum_{j=0}^{k+n+1} \a_j=0. 
\end{align}
We set $\a=(\a_0,\a_1\dots,\a_{k+n},\a_{k+n+1})$. 
When we consider $z\in Z^{(1)}$ with $|z \la J\ra |=0$, 
we also assume the condition 
\begin{align}
  \label{eq:non-int-J}
  \sum_{j\in J} \a_j \not\in \Z . 
\end{align}

We often regard the $\a_i$ as being indeterminate. 
For an element $f(\a)$ of the rational function field 
$\C (\a)=\C (\a_0 ,\ldots ,\a_{k+n+1})$, 
we set $f(\a)^{\vee}=f(-\a)$. 
For a matrix $A$ with entries in $\C(\a)$, 
we write $A^{\vee}$ to denote the matrix 
given by applying ${}^{\vee}$ to each entry of $A$.

\subsection{Twisted homology groups}
\label{subsection:homology-setting}
We fix $z \in Z^{(01)}$ 
and consider the twisted homology groups associated with 
the multivalued functions
\begin{align*}
  u_z(t)=\prod_{j=1}^{k+n+1}L_j (t;z)^{\a_j} ,\quad 
  u_z(t)^{-1}=\frac{1}{u_z(t)}=\prod_{j=1}^{k+n+1}L_j(t;z)^{-\a_j} 
\end{align*}
on $T_z$. 
We denote the $k$-th twisted homology group by $H_k (T_z, u_z)$, 
and the locally finite one by $H^{lf}_k (T_x ,u_z)$. 
It is known that $H^{lf}_k (T_z ,u_z^{\pm 1})$ is isomorphic to 
$H_k (T_z ,u_z^{\pm 1})$; 
we identify these groups and use the notation $H_k (T_z ,u_z^{\pm 1})$. 
We refer to a $k$-dimensional twisted cycle as simply a twisted cycle. 
We also note that the intersection form $\cI^h$ is defined between 
$H_k (T_z ,u_z)$ and $H_k (T_z ,u_z^{-1})$. 

% For $x,x' \in X$ and a path $\tau$ in $X$ from $x$ to $x'$, 
% there is the canonical isomorphism 
% $$
% \tau_* : H_k (T_x ,u_x) \to H_k (T_{x'} ,u_{x'}) .
% $$
% Hence the family 
% $$
% \cH^{\a}=\bigcup_{x \in X} H_k (T_x ,u_x)
% $$
% forms a local system on $X$. 
% Similarly, we obtain a local system 
% $$
% \cH^{-\a}=\bigcup_{x \in X} H_k (T_x ,u_x^{-1}) .
% $$

When we use twisted homology groups and intersection numbers, 
we mainly consider the case when each entry of $z$ is a real number, that is,  
$\cA_z$ is the complexification of a real arrangement. 
We set
\begin{align*}
  Z^{(01)}_{\R} = \{ z\in Z^{(01)} \mid \textrm{each entry of $z$ is a real number} \} .
\end{align*}
Let $\cC (\cA_z)$ be the set of chambers of $\cA_z$ and 
$\bC (\cA_z)$ be the set of compact chambers. 
For a chamber $\De \in \cC (\cA_z)$, we can regard 
$\De \ot u_z^{\pm 1}$ as twisted cycles in $H^{lf}_k (T_z ,u_z^{\pm 1}) =H_k (T_z ,u_z^{\pm 1})$ 
by taking suitable branches of $u_z^{\pm 1}$. 
We say that $\De \ot u_z^{\pm 1}$ are \textit{loaded chambers}. 
% Let $\cC (\cA_z)\ot u_z^{\pm 1}$ (resp. $\bC (\cA_z)\ot u_z^{\pm 1}$) 
% be the subspace of $H_k (T_z ,u_z^{\pm 1})$ spanned by  
% the loaded chanbers (resp. the loaded bounded chambers). 
Formulas to evaluate the intersection numbers of loaded chambers are given in \cite{KY}.
The formulas are simple for $z\in Z^{(0)}$, but  
they are more complicated for $z\in Z^{(1)}$ because of the blowing-up process. 
We present a new method to evaluate them in Section \ref{section:vanishing}. 
\begin{fact}[{\cite{AK}, \cite[Proposition 3.1.4]{DT}}]
  \label{fact-homology-chamber}
  For any $z\in Z^{(01)}_{\R}$, we have 
  \begin{align*}
    \dim H_k (T_z ,u_z^{\pm 1}) =
    \begin{cases}
      \binom{k+n}{k} & (z\in Z^{(0)}) , \\
      \binom{k+n}{k}-1 & (z\in Z^{(1)}) .
    \end{cases}
  \end{align*}
  In particular, if $z\in Z^{(01)}_{\R}$, 
  the loaded bounded chambers 
  $\{ \De \ot u_z^{\pm 1} \mid \De \in \bC (\cA_z) \}$ form 
  bases of the twisted homology groups. 
\end{fact}
This fact implies that 
the number of bounded chambers is $\binom{k+n}{k}$ (resp. $\binom{k+n}{k}-1$)
if $z\in Z^{(0)}$ (resp. $z\in Z^{(1)}$). 

For $\a=(\a_0,\a_1 ,\dots,\a_{k+n},\a_{k+n+1})$, 
we define $\l=(\l_0,\l_1 , \dots,\l_{k+n},\l_{k+n+1})$ by $\l_j =e^{\tpi \a_j}$. 
We also regard the $\l_i$ as being indeterminate. 
For an element $f(\l)$ of the rational function field 
$\C (\l)=\C (\l_0 ,\ldots ,\l_{k+n+1})$, 
we have $f(\l)^{\vee}=f(1/\l_0 ,\ldots ,1/\l_{k+n+1})$. 
According to \cite{KY}, the intersection numbers between loaded chambers 
are valued in $\C(\l)$.

\subsection{Twisted cohomology groups}
\label{section:cohomology}
For a fixed $z \in Z^{(01)}$, 
let $\W^l (T_z)$ be the vector space of rational $l$-forms on $\P^k$ 
with poles only along $\P^k -T_z$.
We set an $1$-form as
\begin{align*}
  \w = \sum_{j=1}^{k+n+1} \a_j \dt \log L_j 
  = \sum_{j=1}^{k+n+1} \a_j \frac{\dt L_j}{L_j} \in \W^1 (T_z),
\end{align*}
where $\dt$ is the exterior derivative with respect to $t_1,\dots,t_k$. 
We consider the twisted cohomology groups 
\begin{align*}
  H^k(\W^\bu(T_z),\na^{\pm \a})=\W^k(T_z)/\na^{\pm \a}(\W^{k-1}(T_z)) ,
\end{align*}
where $\na^\a=\dt+\w\wedge$. 
Note that there is an intersection pairing $\cI^c$ between 
$H^k(\W^\bu(T_x),\na^{\a})$ and $H^k(\W^\bu(T_x),\na^{-\a})$.

From the duality property (e.g., \cite[Lemma 2.9]{AK}) and 
Fact \ref{fact-homology-chamber}, we obtain the dimensions of 
the twisted cohomology groups. 
\begin{fact}
  \begin{align*}
    \dim H^k(\W^\bu(T_z),\na^{\pm\a}) =\left\{
      \begin{array}{ll}
        \binom{k+n}{k} & (z\in Z^{(0)}) , \\
        \binom{k+n}{k}-1 & (z\in Z^{(1)}) .
      \end{array}
    \right.
  \end{align*}
\end{fact}
The bases of $H^k(\W^\bu(T_z),\na^{\pm \a})$ are given by 
the logarithmic $k$-forms. 
For $J=\{ j_0,j_1,\dots,j_l \} \subset \{ 0,1,\ldots ,k+n+1 \}$,  
we set the logarithmic $l$-form as
\begin{align*}
  \f\la J\ra =\dt\log(L_{j_1}/L_{j_0})\wedge \dt\log(L_{j_2}/L_{j_0})
  \wedge \cdots \wedge\dt\log(L_{j_l}/L_{j_0}) .
\end{align*}
We mainly use the case where $\#J =k+1$ (i.e., $J\in \cJ$). 
In such a case, the $k$-form $\f\la J\ra$ is expressed as 
\begin{align}
  \label{eq:phi}
  \f\la J\ra = \frac{|z \la J\ra |}{\prod_{p=0}^k L_{j_p}} \dd t_1\wedge \cdots \wedge\dd t_k 
\end{align}
(see, e.g., \cite[Fact 2.5]{GM-PC}).
There are explicit formulas for the intersection numbers of 
logarithmic $k$-forms. 
\begin{fact}[\cite{M1}]
  \label{fact:intersection}
  We assume $z\in Z^{(0)}$. 
  For $J=\la j_0 \dots j_k \ra$ and $J'=\la j'_0 \dots j'_k \ra$, we have  
  \begin{align*}
    \cI^c(\f\la J\ra,\f\la J'\ra)
    =
    \begin{cases}
      (2\pi\sqrt{-1})^k \cdot 
      \dfrac{\sum_{j\in J}\a_{j}}
      {\prod_{j\in J}\a_{j}}
      &\textrm{if } J=J',\\
      (2\pi\sqrt{-1})^k \cdot 
      \dfrac{(-1)^{p+q}}
      {\prod_{j\in J\cap J'}\a_{j}}
      &\textrm{if } \#(J\cap J')=k,\\
      0&\textrm{otherwise,}
    \end{cases}
  \end{align*}
  where we take $p$ and $q$ such that 
  $J-\{j_p\}=J'-\{j'_q\}$ 
  in the case of $\#(J\cap J')=k$.
\end{fact}
In particular, these intersection numbers belong to $\C(\a)$ and 
satisfy $\cI^c(\f\la J\ra,\f\la J'\ra)=\cI^c(\f\la J'\ra,\f\la J\ra)$. 
Note that if we regard the $\a_i$ as being indeterminate, we have 
$\cI^c(\f\la J\ra,\f\la J'\ra)^{\vee}=(-1)^k \cdot \cI^c(\f\la J\ra,\f\la J'\ra)$. 

By the non-degeneracy of $\cI^c$, we obtain bases of 
$H^k(\W^\bu(T_z),\na^{\pm\a})$ for $z\in Z^{(0)}$.
\begin{fact}[\cite{GM-PC}]
  \label{fact:basis}
  We assume $z\in Z^{(0)}$. 
  Let $p$ and $q$ be two different elements in the set 
  $\{0,1,\dots,k+n+1\}$. We set 
  \begin{align*}
    \under{q}{\cJ}{p}=\{J\in \cJ\mid q\notin J,\ p\in J\}.
  \end{align*}
  Then, $\{\f\la J\ra \mid J\in \under{q}{\cJ}{p} \}$ 
  gives bases of 
  $H^k(\W^\bu(T_z),\na^{\pm\a})$.
\end{fact}

Let $\Phi^l \subset \W^l (T_z)$ be the $\C (\a)$-vector subspace 
spanned by the $\f \la J \ra$ with $\#J =l+1$. 
\begin{fact}[{\cite[Theorem 6]{KitaNoumi}}]
  \label{fact:log-form-iso}
  If $z\in Z^{(0)}$, then we have 
  \begin{align*}
    H^k(\W^\bu(T_z),\na^{\pm\a}) =\Phi^k / (\pm \w \we \Phi^{k-1}) .
  \end{align*}
\end{fact}

Using the argument in \cite[Examples 4.5 and 4.6]{FT}, 
we can take bases of $H^k(\W^\bu(T_z),\na^{\pm\a})$ with $z\in Z^{(1)}$. 
For $J=\la j_0 \dots j_k \ra \in \cJ$ and $q\not\in J$, 
we denote
\begin{align*}
  \under{j_l}{J}{q} =\la j_0 \dots j_{l-1} ~q~ j_{l+1} \dots j_k \ra .
\end{align*}
\begin{fact}[\cite{FT}]
  \label{fact:cohomology-basis-Z1}
  Let $z\in Z^{(1)}$. 
  We take $J^{\circ}=\la j_0 \dots j_k \ra$ such that $|z\la J^{\circ}\ra |=0$. 
  For $l=0,\dots ,k$ and $p,q\not\in J^{\circ}$ with $p\neq q$, 
  each of the sets
  \begin{align}
    \label{eq:beta-nbc1}
    &\{ \f \la J\ra \mid J\in \under{j_l}{\cJ}{p}-\{ \under{j_l}{J^{\circ}}{p}\} \} ,\\
    \label{eq:beta-nbc2}
    &\{ \f \la J\ra \mid J\in \under{q}{\cJ}{p}-\{ \under{j_l}{J^{\circ}}{p}\} \} ,\\
    \label{eq:beta-nbc3}
    &\{ \f \la J\ra \mid J\in \under{q}{\cJ}{j_l}-\{ J^{\circ} \} \} 
  \end{align}
  gives bases of $H^k(\W^\bu(T_z),\na^{\pm\a})$.
\end{fact}
\begin{proof}
  First, we show that (\ref{eq:beta-nbc1}) and (\ref{eq:beta-nbc2}) are bases. 
  By considering a suitable coordinates change if necessary, 
  we may assume that $p=0$ (and hence $0\not\in J^{\circ}$). 
  In this case, the arrangement $\cA_z$ of hyperplanes is 
  in a general position to infinity. 
  According to \cite[Example 4.5]{FT}, 
  we obtain a ``monomial'' basis of $H^k(\W^\bu(T_z),\na^{\pm\a})$ 
  corresponding to  
  a linear order on the index set $\{ 1,2,\dots ,k+n+1\}$. 
  \begin{enumerate}[(i)]
  \item We take a linear order $\prec$ such that 
    \begin{align*}
      {\min}_{\prec} \{ 1,2,\dots ,k+n+1\} 
      ={\min}_{\prec} (J^{\circ}) =j_l , 
    \end{align*}
    where $\min_{\prec}$ implies the minimum with $\prec$. 
    The beta-system $\b\nbc (\cA_z)$ then consists of
    $I=\{i_1 ,\dots ,i_k\}$ such that 
    $I\not\ni j_l$ and $I\neq \{ j_0 ,\dots ,j_{l-1}, j_{l+1} ,\dots, j_k \}$. 
    The set 
    \begin{align}
      \label{eq:beta-nbc1-1}
      \left\{  
        \frac{\dt L_{i_1}}{L_{i_1}} \we \dots \we \frac{\dt L_{i_k}}{L_{i_k}} \mid
        \{ i_1,\dots ,i_k\} \in \b\nbc (\cA_z)
      \right\}
    \end{align}
    gives bases of $H^k(\W^\bu(T_z),\na^{\pm\a})$. 
    On the other hand, we have
    \begin{align*}
      \frac{\dt L_{i_1}}{L_{i_1}} \we \dots \we \frac{\dt L_{i_k}}{L_{i_k}}
      =\dt\log L_{i_1} \we \dots \we \dt\log L_{i_k}
      =\f \la 0 i_1 \dots i_k\ra 
    \end{align*}
    and 
    \begin{align*}
      \{ 0, i_1,\dots ,i_k\} \in \under{j_l}{\cJ}{0}-\{ \under{j_l}{J^{\circ}}{0}\}
      \Leftrightarrow 
      \{ i_1,\dots ,i_k\} \in \b\nbc (\cA_z) .
    \end{align*}
    Therefore, the set (\ref{eq:beta-nbc1}) coincides with (\ref{eq:beta-nbc1-1}) 
    up to the sign. 
  \item We take a linear order $\prec$ such that 
    \begin{align*}
      q={\min}_{\prec} \{ 1,2,\dots ,k+n+1\} 
      \precneqq {\min}_{\prec} (J^{\circ}) =j_l . 
    \end{align*}
    The beta-system $\b\nbc (\cA_z)$ then consists of
    $I=\{i_1 ,\dots ,i_k\}$ such that 
    $I\not\ni q$ and $I\neq \{ j_0 ,\dots ,j_{l-1}, j_{l+1} ,\dots, j_k \}$. 
    Similar to the above, 
    the set 
    \begin{align}
      \label{eq:beta-nbc2-1}
      \left\{  
        \f \la 0 i_1 \dots i_k\ra \mid
        \{ i_1,\dots ,i_k\} \in \b\nbc (\cA_z)
      \right\}
    \end{align}
    gives bases of $H^k(\W^\bu(T_z),\na^{\pm\a})$. 
    Because of
    \begin{align*}
      \{ 0, i_1,\dots ,i_k\} \in \under{q}{\cJ}{0}-\{ \under{j_l}{J^{\circ}}{0}\}
      \Leftrightarrow 
      \{ i_1,\dots ,i_k\} \in \b\nbc (\cA_z) ,
    \end{align*}
    the set (\ref{eq:beta-nbc2}) coincides with (\ref{eq:beta-nbc2-1}) 
    up to the sign. 
  \end{enumerate}
  Next, we consider (\ref{eq:beta-nbc3}). 
  Using a suitable coordinate change if necessary, we may assume that
  $j_l =0$. 
  We can apply the discussion in \cite[Example 4.6]{FT}. 
  As the hyperplane $(L_q =0)$ is generic, 
  the beta-system $\b\nbc (\cA_z)$ defined by the order satisfying
  \begin{align*}
    q={\min}_{\prec} \{ 1,2,\dots ,k+n+1\}
  \end{align*}
  gives a monomial basis. 
  $\b\nbc (\cA_z)$ consists of $I=\{i_1 ,\dots ,i_k\}$ such that 
  $I\not\ni q$ and $I\neq \{ j_0 ,\dots ,j_{l-1}, j_{l+1} ,\dots, j_k \}$. 
  As $j_l=0$, we have 
  \begin{align*}
    \{ 0, i_1,\dots ,i_k\} \in \under{q}{\cJ}{0}-\{ J^{\circ} \}
    \Leftrightarrow 
    \{ i_1,\dots ,i_k\} \in \b\nbc (\cA_z) .
  \end{align*}
  Thus, the monomial basis 
  $\{ \f \la 0 i_1 \dots i_k\ra \mid \{ i_1,\dots ,i_k\} \in \b\nbc (\cA_z)\}$ 
  coincides with (\ref{eq:beta-nbc3}) up to the sign.
\end{proof}

\section{Vanishing (co)cycle and intersection numbers}
\label{section:vanishing}
In this section, we study the twisted homology and cohomology groups 
and intersection forms for $z_0 \in Z^{(1)}$ using 
those for $z \in Z^{(0)}$. 
The arguments for twisted cohomology and homology are parallel.  

We fix $z_0 \in Z^{(1)}$. 
There exists a unique $J\in \cJ$ such that $|z_0\la J\ra| =0$. 
We denote this by $J^{\van}=\la j_0 \dots j_k\ra$, that is, 
\begin{align*}
  |z_0 \la J^{\van} \ra| = 0 ,\qquad 
  |z_0 \la J \ra| \neq 0 \ (J\neq J^{\van}) .
\end{align*}
Let $\e_0 ,\dots ,\e_k$ be sufficiently small positive real numbers 
such that the matrix $z$ obtained by replacing the $j_k$-th column of $z_0$ with 
$\TP (z_{0,j_k}+\e_0 ,\dots ,z_{k,j_k}+\e_k)$ 
belongs to $Z^{(0)}$ (we can take such $\e_j$ because $Z^{(0)}$ is
a Zariski open subset of $Z\simeq \C^{(k+1)\times (k+n+2)}$). 
Then, $z\in Z^{(0)}$ is sufficiently close to $z_0$ in $Z$. 
In other words, $\cA_z$ is obtained as a perturbation of $(L_{j_k}=0) \in \cA_{z_0}$. 

Recall that we have assumed  
\begin{align}
  \label{eq:condition-alpha}
  \a_0,\a_1,\dots,\a_{k+n},\a_{k+n+1} \not\in \Z , \quad 
  \sum_{j=0}^{k+n+1} \a_j=0 , \quad
  \sum_{p\in J^{\van}}\a_p \not\in \Z ,
\end{align}
or, equivalently, 
\begin{align}
  \label{eq:condition-lambda}
  \l_0,\l_1,\dots,\l_{k+n},\l_{k+n+1} \neq 1 , \quad 
  \prod_{j=0}^{k+n+1} \l_j=1 , \quad
  \prod_{p\in J^{\van}}\l_p \neq 1 ,
\end{align}
when we assign $\a_j$ and $\l_j$ to complex numbers. 

\subsection{Twisted cohomology}
First, we consider twisted cohomology groups.
To distinguish the notation for $z_0$ and $z$, 
we use the following:
\begin{align*}
  &\w_0 =\w|_{z=z_0} =\sum_{j=1}^{k+n+1} \a_j \dt \log L_j (t;z_0), \quad 
  \na_0^{\pm \a} = d_t \pm \w_0 \we ,\\
  &\cI^c_0 :
  H^k(\W^\bu(T_{z_0}),\na_0^{\a}) \times H^k(\W^\bu(T_{z_0}),\na_0^{-\a}) \to \C(\a) ,
\end{align*}
whereas $\w$, $\na^{\pm \a}$, $\cI^c$ are used for $H^k(\W^\bu(T_{z}),\na^{\pm \a})$. 

Let $\f^{\van}_{\pm} \in H^k(\W^\bu(T_{z}),\na^{\pm \a})$ be the element expressed by
$\f \la J^{\van}\ra$, which is a \textit{vanishing form} as $z \to z_0$ 
because of (\ref{eq:phi}). 

\begin{definition}[Limit of a twisted cocycle]
  As the denominator of each logarithmic form $\f \la J \ra$ 
  dose not become the zero polynomial in $t_1,\dots ,t_k$ 
  as $z\to z_0$, we can define the \textit{limit} 
  $\lim_{z\to z_0} \f \la J \ra \in \W^l (T_{z_0})$ of 
   $\f \la J \ra \in \W^l (T_{z})$.    
\end{definition}
Note that $\lim_{z\to z_0} \f \la J^{\van} \ra =0 \in \W^k (T_{z_0})$. 
\begin{proposition}
  \label{prop-well-def-cohomology}
  The correspondences
  \begin{align*}
    \fl^c_{\pm} : H^k(\W^\bu(T_{z}),\na^{\pm \a})
    =\Phi^k / (\pm \w \we \Phi^{k-1})
    &\to H^k(\W^\bu(T_{z_0}),\na_0^{\pm \a});\\
    \f &\mapsto \lim_{z\to z_0} \f
  \end{align*}
  are well-defined $\C(\a)$-linear maps. 
\end{proposition}
\begin{proof}
  By Fact \ref{fact:log-form-iso}, it is sufficient to show that
  the image of each $\w \we \f\la j_0 j_1 \cdots j_{k-1} \ra$ 
  is zero in $H^k(\W^\bu(T_{z_0}),\na_0^{\pm \a})$. 
  As $\lim_{z\to z_0}\f\la j_0 j_1 \dots j_{k-1} \ra$ is $\dt$-closed, 
  we have
  \begin{align*}
    \lim_{z\to z_0} (\w \we \f\la j_0 j_1 \dots j_{k-1} \ra)
    &=\w_0 \we \lim_{z\to z_0} \f\la j_0 j_1 \dots j_{k-1} \ra \\
    &=\na_0^{\pm \a} \left( \pm \lim_{z\to z_0} \f\la j_0 j_1 \dots j_{k-1} \ra\right) .
  \end{align*}
  Thus, the claim is proved. 
\end{proof}
We define subspaces of $H^k(\W^\bu(T_{z}),\na^{\pm \a})$ as 
\begin{align*}
  \la \f^{\van}_{\pm}\ra &=\C(\a)\cdot \f^{\van}_{\pm} 
  \subset H^k(\W^\bu(T_{z}),\na^{\pm \a}), \\
  \la \f^{\van}_{-}\ra^{\perp} &= \{ \f \in H^k(\W^\bu(T_{z}),\na^{\a}) \mid 
  \cI^c (\f ,\f^{\van}_{-})=0 \} , \\
  \la \f^{\van}_{+}\ra^{\perp} &= \{ \f \in H^k(\W^\bu(T_{z}),\na^{-\a}) \mid 
  \cI^c (\f^{\van}_{+} ,\f)=0 \} .
\end{align*}
By Fact \ref{fact:intersection} and the assumption in (\ref{eq:condition-alpha}), 
we have 
\begin{align}
  \label{eq:self-intersection-van-cohomology}
  \cI^c (\f^{\van}_{+} ,\f^{\van}_{-}) 
  = (2\pi\sqrt{-1})^k \cdot 
  \dfrac{\sum_{j\in J^{\van}}\a_{j}}
  {\prod_{j\in J^{\van}}\a_{j}} \neq 0,  
\end{align}
and hence $\f^{\van}_{\pm} \not\in \la \f^{\van}_{\mp}\ra^{\perp}$. 
This implies that $\la \f^{\van}_{\mp}\ra^{\perp} \subsetneq H^k(\W^\bu(T_{z}),\na^{\pm \a})$. 
We take $p\not\in J^{\van} =\{ j_0 ,\dots  ,j_k \}$, 
$l\in \{ 0 ,\dots ,k\}$, and 
set $\cJ^{\perp}=\under{j_l}{\cJ}{p} -\{ \under{j_l}{J^{\van}}{p}\}$. 
Because 
\begin{align}
  \label{eq:J-Jvan}
  \# (J\cap J^{\van} )<k, \quad J\in \cJ^{\perp} ,  
\end{align}
$\{ \f\la J\ra \mid J\in \cJ^{\perp} \}$ expresses elements in $\la \f^{\van}_{\mp}\ra^{\perp}$.
As the $\{ \f\la J\ra \mid J\in \cJ^{\perp}\}$ are linearly independent 
because of Fact \ref{fact:basis}, 
% and 
% $\la \f^{\van}_{\mp}\ra^{\perp}$ do not coinside with $H^k(\W^\bu(T_{z}),\na^{\pm \a})$, 
we have $\dim \la \f^{\van}_{\mp}\ra^{\perp} =\binom{k+n}{k}-1$ with
bases formed by $\{ \f\la J\ra \mid J\in \cJ^{\perp}\}$. 
Thus, we obtain the direct sum decompositions 
\begin{align}
  \label{eq:d-sum-cohomology}
  H^k(\W^\bu(T_{z}),\na^{\pm \a})=\la \f^{\van}_{\pm}\ra \oplus \la \f^{\van}_{\mp}\ra^{\perp} 
\end{align}
and the projections onto the second components
\begin{align}
  \label{eq:2nd-proj-cohomology-1}
  \fp^c_{+} : H^k(\W^\bu(T_{z}),\na^{\a}) \to \la \f^{\van}_{-}\ra^{\perp} ;\quad &
  \f \mapsto \f -\frac{\cI^c(\f,\f^{\van}_{-})}{\cI^c(\f^{\van}_{+},\f^{\van}_{-})}\f^{\van}_{+} ,\\
  \nonumber
  \fp^c_{-} : H^k(\W^\bu(T_{z}),\na^{-\a}) \to \la \f^{\van}_{+}\ra^{\perp} ;\quad &
  \f \mapsto \f -\frac{\cI^c(\f^{\van}_{+},\f)}{\cI^c(\f^{\van}_{+},\f^{\van}_{-})}\f^{\van}_{-} .
\end{align}
By $\lim_{z\to z_0}\f\la J^{\van}\ra =0$, the images of $\f^{\van}_{\pm}$ 
under $\fl^c_{\pm}$ are zero. Hence, we obtain the following commutative diagrams. 
\begin{align}
  \label{diagram-cohomology}
  \xymatrix{
    &H^k(\W^\bu(T_{z}),\na^{\pm \a}) \ar[dl]_{\fp^c_{\pm}}\ar[dr]^{\fl^c_{\pm}}& \\
    \la \f^{\van}_{\mp}\ra^{\perp} \ar[rr]_{\fl^c_{\pm}|_{\la \f^{\van}_{\mp}\ra^{\perp} }} 
    && H^k(\W^{\bu}(T_{z_0}),\na^{\pm \a}_0)
  }
\end{align}
\begin{theorem}\label{th-cohomology-iso}
  The $\C(\a)$-linear maps
  \begin{align*}
    \fl^c_{\pm}|_{\la \f^{\van}_{\mp}\ra^{\perp}} :
    \la \f^{\van}_{\mp}\ra^{\perp} \to H^k(\W^{\bu}(T_{z_0}),\na^{\pm \a}_0)
  \end{align*}
  are isomorphisms that preserve the intersection form, that is, 
  \begin{align*}
    \cI^c_0 (\fl^c_{+}(\f),\fl^c_{-}(\f'))
    =\cI^c (\f, \f') ,\quad 
    \f \in \la \f^{\van}_{-}\ra^{\perp} ,\ \f' \in \la \f^{\van}_{+}\ra^{\perp} .
  \end{align*}
\end{theorem}
\begin{proof}
  Recall that $\cJ^{\perp}=\under{j_l}{\cJ}{p} -\{ \under{j_l}{J^{\van}}{p}\}$. 
  As the sets
  $\{\fl^c_{\pm}(\f\la J\ra ) \mid J\in \cJ^{\perp}\}$ of logarithmic forms 
  give bases of $H^k(\W^{\bu}(T_{z_0}),\na^{\pm \a}_0)$
  by Fact \ref{fact:cohomology-basis-Z1}, 
  the maps $\fl^c_{\pm}|_{\la \f^{\van}_{\mp}\ra^{\perp}}$ are isomorphisms. 
  Thus, it suffices to show that 
  \begin{align}
    \label{eq:int-preserve-basis}
    \cI^c_0 (\fl^c_{+}(\f\la J\ra ),\fl^c_{-}(\f\la J'\ra))
    =\cI^c (\f\la J\ra, \f\la J'\ra) ,\quad 
    J,J' \in \cJ^{\perp}. 
  \end{align}
  To evaluate $\cI^c_0$ for logarithmic forms, we should consider 
  the blow-up at the point $\bigcap_{p\in J^{\van}} (L_p =0)$, 
  so that the pole divisor of the pull-back of $\w_0$ is a normal crossing. 
  % By $\# (J\cap J^{\van} )<k$ for $J\in \cJ^{\perp}$, 
  By (\ref{eq:J-Jvan}), 
  it follows that 
  the blow-up does not influence the evaluation of 
  $\cI^c_0 (\fl^c_{+}(\f\la J\ra ),\fl^c_{-}(\f\la J'\ra))$. 
  Therefore, we obtain claim (\ref{eq:int-preserve-basis}). 
  % We prove that $C=(\cI^c (\f\la J\ra, \f\la J'\ra))_{J,J'\in \cJ^{\perp}}$ 
  % is invertible. 
  % This and non-degeneracy of $\cI^c_0$ 
  % imply that $\{\fl^c_{\pm}(\f\la J\ra )\}_{J\in \cJ^{\perp}}$
  % are linealy independent in $H^k(\W^{\bu}(T_{z_0}),\na^{\pm \a}_0)$, 
  % and hence they form bases of $H^k(\W^{\bu}(T_{z_0}),\na^{\pm \a}_0)$.  
  % Further, (\ref{eq:int-preserve-basis}) shows that $\fl^c_{\pm}$ preserve the intersection form. 
  % Now, we show the claim. 
  % Let $a=\tr (\ldots ,a_J ,\ldots)_{J\in \cJ^{\perp}}$ be a column vector 
  % that satisfies $Ca=0$. 
  % This implies that 
  % $$
  % \cI^c \Bigl( \f\la J\ra, \sum_{J'\in \cJ^{\perp}} a_{J'}\f\la J'\ra \Bigr) =0 ,\quad 
  % J\in \cJ^{\perp} .
  % $$
  % Since $\{ \f\la J\ra \}_{J\in \cJ^{\perp}}$ forms a basis of $\la \f^{\van}_{-}\ra^{\perp}$, 
  % it means 
  % $$
  % \cI^c \Bigl( \f, \sum_{J'\in \cJ^{\perp}} a_{J'}\f\la J'\ra \Bigr) =0 ,\quad 
  % \f \in \la \f^{\van}_{-}\ra^{\perp} .
  % $$
  % On the other hand, we have 
  % $$
  % \cI^c \Bigl( \f^{\van}_{+}, \sum_{J'\in \cJ^{\perp}} a_{J'}\f\la J'\ra \Bigr) =0 .
  % $$
  % The decomposition (\ref{eq:d-sum-cohomology}) and non-degenerady of $\cI^c$ 
  % yield that $\sum_{J'\in \cJ^{\perp}} a_{J'}\f\la J'\ra =0$, 
  % that is, $a=0$. 
  % We thus conclude that $C$ is invertible. 
\end{proof}
\begin{cor}\label{cor-intesection-cohomology}
  If $\f \in H^k(\W^{\bu}(T_{z_0}),\na^{\a}_0)$ and 
  $\psi \in H^k(\W^{\bu}(T_{z_0}),\na^{-\a}_0)$ 
  are expressed as 
  $\f =\fl^c_{+}(\f')$ and $\psi =\fl^c_{-}(\psi')$ with 
  $\f' \in H^k(\W^\bu(T_{z}),\na^{\a})$, 
  $\psi' \in H^k(\W^\bu(T_{z}),\na^{-\a})$, 
  then the intersection number $\cI^c_0 (\f, \psi)$ is 
  expressed by those on $H^k(\W^\bu(T_{z}),\na^{\pm \a})$: 
  \begin{align*}
    \cI^c_0 (\f, \psi) =\cI^c (\f' ,\psi')
    -\frac{\cI^c(\f', \f^{\van}_{-}) \cdot \cI^c(\f^{\van}_{+},\psi')}{\cI^c(\f^{\van}_{+},\f^{\van}_{-})} .
  \end{align*}
\end{cor}
The right-hand side consists of $\cI^c$. 
When we use this formula, we do not need the blow-up. 
\begin{proof}
  By Theorem \ref{th-cohomology-iso} and 
  diagram (\ref{diagram-cohomology}), we have 
  \begin{align*}
    \cI^c_0 (\f, \psi) &=\cI^c_0 (\fl^c_{+}(\f') ,\fl^c_{-}(\psi'))
    =\cI^c_0 (\fl^c_{+}\circ \fp^c_{+}(\f') ,\fl^c_{-} \circ \fp^c_{-}(\psi')) \\
    &=\cI^c (\fp^c_{+}(\f') ,\fp^c_{-}(\psi')) . 
  \end{align*}
  The expressions for $\fp^c_{\pm}$ in (\ref{eq:2nd-proj-cohomology-1}) 
  and the bilinearity of $\cI^c$ yield the formula. 
\end{proof}

\subsection{Twisted homology}
Next, we consider twisted homology groups.
As mentioned in Subsection \ref{subsection:homology-setting}, 
we assume that $z,z_0 \in Z^{(01)}_{\R}$. 
Similar to the previous subsection, 
we use the notation 
\begin{align*}
  \cI^h &:
  H_k(T_{z},u_{z}) \times H_k(T_{z},u_{z}^{-1}) \to \C(\l) ,\\ 
  \cI^h_0 &:
  H_k(T_{z_0},u_{z_0}) \times H_k(T_{z_0},u_{z_0}^{-1}) \to \C(\l) . 
\end{align*}

First, we discuss the cases when $0=j_0\in J^{\van}$. 
Namely, $(L_{j_1}(t;z_0)=0)$,$\dots$, $(L_{j_k}(t;z_0)=0) \in \cA_{z_0}$ intersect 
in the hyperplane $(L_0(t;z_0)=0)$ at infinity. 

For a subset $\{ l_1 ,\dots ,l_k\} \subset \{ 1,\dots ,k+n+1\}$, 
% not containing $0$, 
we partition the matrix $z\la 0 l_1 \dots l_k\ra$ into blocks as
\begin{align*}
  z\la 0 l_1 \dots l_k\ra =
  \begin{pmatrix}
    1 & z\la l_1 \dots l_k\ra_{0} \\
    0 & \\
    \vdots & z\la l_1 \dots l_k\ra_{1} \\
    0 & 
  \end{pmatrix}, 
\end{align*}
where $z\la l_1 \dots l_k\ra_{0}$ is a row vector of size $k$ and 
$z\la l_1 \dots l_k\ra_{1}$ is a square matrix of size $k$. 
As $z\la 0 l_1 \dots l_k\ra$ is invertible, 
$z\la l_1 \dots l_k\ra_{1}$ is also invertible and we have 
\begin{align*}
  z\la 0 l_1 \dots l_k\ra^{-1} =
  \begin{pmatrix}
    1 & -z\la l_1 \dots l_k\ra_{0} \cdot z\la l_1 \dots l_k\ra_{1}^{-1} \\
    0 & \\
    \vdots & z\la l_1 \dots l_k\ra_{1}^{-1} \\
    0 & 
  \end{pmatrix}. 
\end{align*}
It is easy to see that 
the (affine) coordinates of the intersection point 
$(L_{l_1}(t;z)=0)\cap \cdots \cap (L_{l_k}(t;z)=0)$ are expressed as
\begin{align*}
  (t_1 ,\dots ,t_k) = -z\la l_1 \dots l_k\ra_{0} \cdot z\la l_1 \dots l_k\ra_{1}^{-1}. 
\end{align*}
We denote this point as $P(l_1 ,\dots ,l_k) \in \R^k$. 
Note that when $z\to z_0$, $P(J^{\van}) =P(j_1 ,\dots ,j_k)$ goes to the hyperplane at infinity
($z_0\la j_1 \dots j_k\ra_{1}$ is not invertible). 
As $z$ is sufficiently close to $z_0$, we may assume that 
the norm of $P(J^{\van})$ is much greater than those of the other $P(l_1 ,\dots ,l_k)$.

% There is a chamber $\De^{\van} \subset T_z \cap \R^k$ 
% surrounded by $(L_{j_1} =0)$,$\dots$, $(L_{j_k} =0)$.  
% This is a (unbounded) $k$-simplex with the vertex $P(J^{\van})$, 
% because $(L_j=0)\cap \De^{\van} =\emptyset$ ($j\not\in J^{\van}$) by the condition for norms of 
% $P(l_1 ,\dots ,l_k)$. 
There is a chamber $\De^{\van} \subset T_z \cap \R^k$ 
surrounded by $(L_{j_1} =0)$,$\dots$, $(L_{j_k} =0)$ 
such that $(L_j=0)\cap \De^{\van} =\emptyset$ ($j\not\in J^{\van}$).  
This is a (unbounded) $k$-simplex with the vertex $P(J^{\van})$. 
We can regard $\De^{\van}$ as a cone with origin $P(J^{\van})$. 
Note that when $z\to z_0$, this chamber vanishes. 
\begin{lemma}
  We set 
  $$
  (\De^{\van})^{\perp} = \{ \De \in \bC (\cA_z) \mid 
  \overline{\De^{\van}}\cap \overline{\De} =\emptyset \} ,
  $$
  where $\overline{\De}$ is the closure of $\De \subset \R^k$. 
  The cardinality of $(\De^{\van})^{\perp}$ is then equal to 
  $\binom{k+n}{k}-1$. 
\end{lemma}
\begin{proof}
  As the cardinality of $\bC (\cA_z)$ is 
  $\binom{k+n}{k}$,  
  it is sufficient to show that 
  there exists a unique chamber in $\bC (\cA_z) -(\De^{\van})^{\perp}$. 
  We consider the chambers $D \in \cC(\cA_z)$ that satisfy 
  $\overline{\De^{\van}}\cap \overline{D} \neq \emptyset$. 
  Note that this intersection is a (non-trivial) face of $\De^{\van}$. 
  If $1\leq d \leq k-1$, any $d$-face of $\De^{\van}$ includes
  a half-line with the initial point $P(J^{\van})$, 
  and hence it is unbounded. 
  Thus, the intersection 
  $\overline{\De^{\van}}\cap \overline{D}$ is unbounded if its dimension 
  is greater than or equal to $1$. 
  Therefore, we obtain the unique chamber $D'\in \cC(\cA_z)$ 
  such that $\overline{\De^{\van}}\cap \overline{D'}$
  is the (unique) vertex $P(J^{\van})$ of $\De^{\van}$. 
  When we regard $\De^{\van}$ as a cone, 
  its opposite $\De^{\van}_{\textrm{op}}$ contains the chamber $D'$. 
  We prove that $D'$ is bounded. 
  We take any $p\not\in J^{\van}$, and consider  
  $\under{j_l}{J^{\van}}{p}$ for $l=1,2,\dots,k$. 
  As $(L_p=0)\cap \De^{\van} =\emptyset$, 
  each $P(\under{j_l}{J^{\van}}{p})$ lies on $\overline{\De^{\van}_{\textrm{op}}}$. 
  Therefore, $D'$ is included in the convex hull of the $k+1$ points
  \begin{align*}
    P(J^{\van}), \quad P(\under{j_l}{J^{\van}}{p}) \ (l=1,2,\dots,k) ,
  \end{align*}
  which is a bounded $k$-simplex. 
  Therefore, we have $D'\in \bC(\cA_x)$, which completes the proof. 
\end{proof}

We denote the twisted cycle defined as $\De^{\van}$ loading $u_z(t)^{\pm 1}$
by $\De^{\van}_{\pm} \in H_k (T_z ,u_z^{\pm 1})$. 
We call these \textit{vanishing cycles}. 
Let $(\De^{\van})^{\perp}=\{ \De^1 ,\dots ,\De^{r-1} \}$ with $r=\binom{k+n}{k}$. 
We also set $\De^j_{\pm} =\De^j \ot u_z^{\pm 1} \in H_k (T_z ,u_z^{\pm 1})$. 
\begin{lemma}
  \label{lem:van-perp}
  Under the assumption in (\ref{eq:condition-lambda}), 
  $\{ \De^1_{\pm},\dots, \De^{r-1}_{\pm} ,\De^{\van}_{\pm} \}$
  % \begin{align*}
  %   \{ \De^{\van}_{\pm} \} \cup 
  %   \{ \De \ot u_z^{\pm 1} \mid \De \in (\De^{\van})^{\perp} \}
  % \end{align*}
  form bases of $H_k (T_z , u_z^{\pm 1})$. 
\end{lemma}
\begin{proof}
  We show the claim for only the ``$+$'' case
  (that for ``$-$'' can be proved in a similar manner). 
  By Fact \ref{fact-homology-chamber}, 
  the $\{ \De^1_{+},\dots, \De^{r-1}_{+} \}$ are linearly independent. 
  We have 
  \begin{align}
    \label{eq:self-intersection-van-homology}
    &\cI^h (\De^{\van}_{+},\De^{\van}_{-})
      \!=\! (-1)^k \frac{\prod_{j\in J^{\van}} \l_j -1}{\prod_{j\in J^{\van}} (\l_j -1)}
      \!=\! \frac{1- \prod_{j\in J^{\van}} \l_j }{\prod_{j\in J^{\van}} (1- \l_j)}\neq 0 ,\\
    \label{eq:perp-homology}
    &\cI^h (\De^{i}_{+},\De^{\van}_{-}) =0 \quad 
      (i=1,\dots ,r-1) 
  \end{align}
  by \cite[II--p.177]{KY}, (\ref{eq:condition-lambda}), and 
  the definition of $(\De^{\van})^{\perp}$. 
  These imply that 
  the $\{ \De^1_{\pm},\dots, \De^{r-1}_{\pm} ,\De^{\van}_{\pm} \}$ are linearly independent. 
\end{proof}

Any chamber $\De \in \cC(\cA_{z})$ is expressed by 
a system of linear inequalities of the form $L_j (t;z) \gtrless 0$. 
By replacing $z$ with $z_0$ in these inequalities, 
we obtain a chamber $\De_0 \in \cC(\cA_{z_0})$. 
Note that $\De^{\van}$ becomes the empty set. 
As $z$ is sufficiently close to $z_0$, 
a chamber $\De \in \cC(\cA_{z})$ whose closure does not contain $P(J^{\van})$ 
hardly changes. 
Then, by the proof of Lemma \ref{lem:van-perp}, 
we have
\begin{align}
  \label{eq:bC-z0}
  \bC (\cA_{z_0}) 
  =\{ \De_0 \mid \De \in (\De^{\van})^{\perp}\}
  =\{ \De^1_0 ,\dots ,\De^{r-1}_0 \} .
\end{align}
Note that $D'_0$ is an unbounded chamber in $T_{z_0}\cap \R^k$. 

\begin{definition}[Limit of a loaded chamber]
  \label{def:limit-homology}
  Let $\De \in \cC(\cA_{z})$ be a chamber that is not $\De^{\van}$. 
  We define the \textit{limits} of 
  loaded chambers $\De \ot u_z^{\pm 1}$ as 
  \begin{align*}
    \lim_{z\to z_0} (\De \ot u_z^{\pm 1}) =\De_0 \ot u_{z_0}^{\pm 1} 
    \in H_k(T_{z_0},u_{z_0}^{\pm 1}),
  \end{align*}
  where the branches of $u_{z_0}(t)^{\pm 1}$ on $\De_0$ are 
  naturally determined by $u_{z_0}(t)=\lim_{z\to z_0}u_z(t)$. 
  For the vanishing cycles $\De^{\van}_{\pm} =\De^{\van} \ot u_{z}^{\pm 1}$, 
  we define $\lim_{z\to z_0} \De^{\van}_{\pm} =0 \in H_k(T_{z_0},u_{z_0}^{\pm 1})$. 
\end{definition}
\begin{lemma}
  \label{lem:intersection-preserve}
  If $\De,\De'\in \cC (\cA_z)$ satisfy 
  $\overline{\De} \cap \overline{\De'} \cap \overline{\De^{\van}}=\emptyset$, then 
  \begin{align*}
    \cI^h_0 \left( \lim_{z\to z_0} (\De \ot u_z) ,
      \lim_{z\to z_0} (\De' \ot u_z^{-1}) \right)
    =\cI^h \left( \De \ot u_z , \De' \ot u_z^{-1} \right) .
  \end{align*}
\end{lemma}
\begin{proof}
  According to \cite{KY}, the intersection numbers between loaded chambers are 
  determined by information about the topological intersection and 
  differences among branches of $u_z$. 

  By the assumption, $\overline{\De} \cap \overline{\De'}$ does not 
  contain $P(J^{\van})$. 
  Thus, the topological intersection of $\overline{\De}$ and $\overline{\De'}$
  (more precisely, that of the support of $\textrm{reg}(\De)$ and $\De'$) 
  hardly changes. 
  As the branches of $u_z^{\pm 1}$ change continuously, 
  we obtain the lemma. 
\end{proof}

\begin{definition}
  As $\{ \De^1_{\pm},\dots, \De^{r-1}_{\pm} ,\De^{\van}_{\pm} \}$ 
  form bases of $H_k(T_z,u_z^{\pm 1})$, we can obtain $\C(\l)$-linear maps 
  $\fl^h_{\pm} : H_k(T_z,u_z^{\pm 1}) \to H_k(T_{z_0},u_{z_0}^{\pm 1})$ 
  by defining the images of the bases as their limits. 
  % \begin{align*}
  %   \fl^h_{\pm} (\De^{\van}_{\pm}) =0,\quad 
  %   \fl^h_{\pm} (\De^i_{\pm} ) =\lim_{z\to z_0} (\De^i\ot u_z^{\pm 1}) .
  % \end{align*}
\end{definition}

By the following proposition, we can interpret 
$\fl^h_{\pm}$ as giving the limits for loaded chambers. 
\begin{proposition}
  \label{prop-well-def-homology}
  For any $\De \in \cC (\cA_z)$, we have 
  \begin{align*}
    \fl^h_{\pm} \left( \De \ot u_z^{\pm 1} \right) = \lim_{z\to z_0} (\De \ot u_z^{\pm 1}) .
  \end{align*}
\end{proposition}
\begin{proof}
  Before proving the proposition, 
  we consider the intersection matrix with respect to the bases 
  $\{ \De^1_{\pm},\dots, \De^{r-1}_{\pm} ,\De^{\van}_{\pm} \}$. 
  We set $\De^r_{\pm} = \De^{\van}_{\pm}$ and 
  consider the intersection matrix $H=(\cI^h (\De^i_{+},\De^j_{-}))_{i,j=1,\dots ,r}$. 
  By (\ref{eq:perp-homology}), $H$ can be written as 
  \begin{align*}
    H=
    \begin{pmatrix}
      &&& 0 \\
      &H'&&\vdots \\
      &&& 0 \\
      0&\cdots &0 & \cI^h (\De^{\van}_{+},\De^{\van}_{-})
    \end{pmatrix},
  \end{align*}
  where $H'$ is a square matrix of size $r-1$. 
  As $\cI^h$ is non-degenerate and 
  $\{ \De^1_{+},\dots, \De^{r}_{+}\}$ is a basis, $H$ is invertible. 
  Thus, $H'$ is also invertible by (\ref{eq:self-intersection-van-homology}). 
  Note that $H'$ can be written as 
  $H'=(\cI^h_0 (\fl^h_{+}(\De^i_{+}),\fl^h_{-}(\De^j_{-})))_{i,j=1,\dots ,r-1}$ 
  using Lemma \ref{lem:intersection-preserve}. 

  We now prove the proposition for $\fl^h_{+}$ 
  (the claim for $\fl^h_{-}$ is proved in the same manner). 
  Using the basis $\{ \De^1_{+},\dots, \De^{r-1}_{+} ,\De^{\van}_{+} \}$, 
  we can write $\De \ot u_z$ as 
  $\De \ot u_z =\sum_{i=1}^{r-1}c_i \De^i_{+} +c_r \De^{\van}_{+}$ with $c_i\in \C(\l)$. 
  It suffices to show that 
  \begin{align*}
    \lim_{z\to z_0} (\De \ot u_z ) =\sum_{i=1}^{r-1}c_i \fl^h_{+}(\De^i_{+}) .
  \end{align*}
  On the other hand, 
  because $\{ \fl^h_{+}(\De^1_{+}),\dots, \fl^h_{+}(\De^{r-1}_{+}) \}$ forms a basis of 
  $H_k(T_{z_0},u_{z_0})$ by Fact \ref{fact-homology-chamber} and (\ref{eq:bC-z0}),  
  $\lim_{z\to z_0} (\De \ot u_z )$ can be expressed as 
  \begin{align*}
    \lim_{z\to z_0} (\De \ot u_z ) =\sum_{i=1}^{r-1}c'_i \fl^h_{+}(\De^i_{+}) 
  \end{align*}
  for some $c'_i \in \C(\l)$. 
  Our claim is reduced to $c_i =c'_i$ ($i=1,\dots ,r-1$). 
  Because of  
  $\overline{\De} \cap \overline{\De^{i}} \cap \overline{\De^{\van}}=\emptyset$, 
  we have 
  \begin{align*}
    \cI^h_0 \left( \lim_{z\to z_0} (\De \ot u_z) , \fl^h_{-}(\De^j_{-}) \right)
    =\cI^h \left( \De \ot u_z , \De^j_{-} \right) .
  \end{align*}
  Therefore, we obtain 
  \begin{align*}
    (c_1,\dots ,c_{r-1}) H' 
    &=\Big( \cI^h (\De \ot u_z,\De^1_{-}) ,\dots , \cI^h (\De \ot u_z,\De^{r-1}_{-}) \Big) \\
    &=\Big( \cI^h_0 \big( \lim_{z\to z_0} (\De \ot u_z) , \fl^h_{-}(\De^1_{-})\big) ,\dots , 
    \cI^h_0 \big( \lim_{z\to z_0} (\De \ot u_z) , \fl^h_{-}(\De^{r-1}_{-})\big) \Big) \\
    &=(c'_1,\dots ,c'_{r-1}) H' . 
  \end{align*}
  The invertibility of $H'$ proves the claim. 
\end{proof}
We define subspaces of $H_k(T_{z},u_z^{\pm 1})$ as 
\begin{align*}
  \la \De^{\van}_{\pm}\ra &=\C(\l)\cdot \De^{\van}_{\pm} 
  \subset H_k(T_{z},u_z^{\pm 1}), \\
  \la \De^{\van}_{-}\ra^{\perp} &= \{ \si \in H_k(T_{z},u_z) \mid 
  \cI^h (\si ,\De^{\van}_{-})=0 \} , \\
  \la \De^{\van}_{+}\ra^{\perp} &= \{ \si \in H_k(T_{z},u_z^{-1}) \mid 
  \cI^h (\De^{\van}_{+} ,\si)=0 \} .
\end{align*}
According to the proof of Lemma \ref{lem:van-perp}, we have
\begin{align*}
  \{ \De^1_{\pm},\dots, \De^{r-1}_{\pm}  \} \subset \la \De^{\van}_{\mp}\ra^{\perp} ,\qquad 
  \De^{\van}_{\pm} \not\in \la \De^{\van}_{\mp}\ra^{\perp} ,
\end{align*}
and hence the dimensions of $\la \De^{\van}_{\mp}\ra^{\perp}$ are $r-1=\binom{k+n}{k}-1$. 
Therefore, we obtain the direct sum decompositions 
\begin{align}
  \label{eq:d-sum-homology}
  H_k(T_{z},u_z^{\pm 1}) =\la \De^{\van}_{\pm}\ra \oplus \la \De^{\van}_{\mp}\ra^{\perp} 
\end{align}
and the projections onto the second components
\begin{align}
  \label{eq:2nd-proj-homology-1}
  \fp^h_{+} : H_k(T_{z},u_z) \to \la \De^{\van}_{-}\ra^{\perp} ;\quad &
  \si \mapsto \si -\frac{\cI^h(\si,\De^{\van}_{-})}
  {\cI^h(\De^{\van}_{+},\De^{\van}_{-})}\De^{\van}_{+} ,\\
  % \label{eq:2nd-proj-homology-2}
  \nonumber
  \fp^h_{-} : H_k(T_{z},u_z^{- 1}) \to \la \De^{\van}_{+}\ra^{\perp} ;\quad &
  \si \mapsto \si -\frac{\cI^h(\De^{\van}_{+},\si)}
  {\cI^h(\De^{\van}_{+},\De^{\van}_{-})}\De^{\van}_{-} .
\end{align}
By the definitions of $\fl^h_{\pm}$, the images of $\De^{\van}_{\pm}$ 
under $\fl^h_{\pm}$ are zero. Hence, we obtain the following commutative diagrams. 
\begin{align}
  \label{diagram-homology}
  \xymatrix{
    &H_k(T_{z},u_z^{\pm 1}) \ar[dl]_{\fp^h_{\pm}}\ar[dr]^{\fl^h_{\pm}} & \\
    \la \De^{\van}_{\mp}\ra^{\perp} \ar[rr]_{\fl^h_{\pm}|_{\la \De^{\van}_{\mp}\ra^{\perp} }} 
    && H_k(T_{z_0},u_{z_0}^{\pm 1})
  }
\end{align}
\begin{theorem}\label{th-homology-iso}
  The $\C(\l)$-linear maps
  \begin{align*}
    \fl^h_{\pm}|_{\la \De^{\van}_{\mp}\ra^{\perp}} :
    \la \De^{\van}_{\mp}\ra^{\perp} \to H_k(T_{z_0},u_{z_0}^{\pm 1})
  \end{align*}
  are isomorphisms that preserve the intersection form, that is, 
  \begin{align*}
    \cI^h_0 (\fl^h_{+}(\si),\fl^h_{-}(\tau))
    =\cI^h (\si, \tau)    
  \end{align*}
  holds for any $\si \in \la \De^{\van}_{-}\ra^{\perp}$, 
  $\tau \in \la \De^{\van}_{+}\ra^{\perp}$.
\end{theorem}
\begin{proof}
  By the above argument, 
  $\{ \De^1_{\pm},\dots, \De^{r-1}_{\pm}  \}$ and 
  $\{ \fl^h_{\pm}(\De^1_{\pm}),\dots, \fl^h_{\pm}(\De^{r-1}_{\pm}) \}$
  form bases of 
  $\la \De^{\van}_{\mp}\ra^{\perp}$ and $H_k(T_{z_0},u_{z_0}^{\pm 1})$, respectively. 
  Along with Lemma \ref{lem:intersection-preserve}, this implies the result of the theorem. 
\end{proof}
Similar to Corollary \ref{cor-intesection-cohomology}, 
we obtain the following formula. 
\begin{cor}\label{cor-intesection-homology}
  If $\si \in H_k(T_{z_0},u_{z_0})$ and 
  $\tau \in H_k(T_{z_0},u_{z_0}^{- 1})$ 
  are expressed as 
  $\si =\fl^c_{+}(\si')$ and $\tau =\fl^c_{-}(\tau')$ with 
  $\si' \in H_k(T_{z},u_z)$, $\tau' \in H_k(T_{z},u_z^{- 1})$, 
  then the intersection number $\cI^h_0 (\si, \tau)$ is 
  expressed by those on $H_k(T_{z},u_z^{\pm 1})$: 
  \begin{align*}
    \cI^h_0 (\si, \tau) =\cI^h (\si' ,\tau')
    -\frac{\cI^h(\si', \De^{\van}_{-}) \cdot \cI^h(\De^{\van}_{+},\tau')}
    {\cI^h(\De^{\van}_{+},\De^{\van}_{-})} .
  \end{align*}
\end{cor}

Next, we consider the case of $0\not\in J^{\van}$. 
In this instance, we can reduce the arguments to those for $0\in J^{\van}$ 
using an appropriate projective transformation on $\P^k$. 
Therefore, the above results are valid in this case. 
Though we should consider a projective transformation to prove the claims,  
we do not need this to state the conclusions. 
Thus, we summarize the results. 
\begin{cor}
  \label{cor-intesection-homology-2}
  Let $\De^{\van}$ be a chamber in $T_z\cap \R^k$
  surrounded by $(L_{j_0} =0)$, $(L_{j_1} =0)$, $\dots$, $(L_{j_k} =0)$ 
  (note that the chamber is bounded if and only if $0\not\in J^{\van}$). 
  We set $\De^{\van}_{\pm} =\De^{\van} \ot u_z^{\pm 1} \in H_k (T_z ,u_z^{\pm 1})$ and  
  \begin{align*}
    \la \De^{\van}_{\pm}\ra &=\C(\l)\cdot \De^{\van}_{\pm} 
    \subset H_k(T_{z},u_z^{\pm 1}), \\
    \la \De^{\van}_{-}\ra^{\perp} &= \{ \si \in H_k(T_{z},u_z) \mid 
    \cI^h (\si ,\De^{\van}_{-})=0 \} , \\
    \la \De^{\van}_{+}\ra^{\perp} &= \{ \si \in H_k(T_{z},u_z^{-1}) \mid 
    \cI^h (\De^{\van}_{+} ,\si)=0 \} .
  \end{align*}
  Then, we have the following. 
  \begin{enumerate}[(i)]
  \item We can define $\fl^h_{\pm} : H_k(T_z,u_z^{\pm 1}) \to H_k(T_{z_0},u_{z_0}^{\pm 1})$ 
    by the limits of loaded chambers in Definition \ref{def:limit-homology}. 
  \item We obtain the direct sum decompositions 
    $H_k(T_{z},u_z^{\pm 1}) =\la \De^{\van}_{\pm}\ra \oplus \la \De^{\van}_{\mp}\ra^{\perp}$,  
    and the projections $\fp^h_{\pm}$ onto the second components 
    are given by (\ref{eq:2nd-proj-homology-1}).
    % and (\ref{eq:2nd-proj-homology-2}). 
  \item The diagrams (\ref{diagram-homology}) are commutative. 
  \item Theorem \ref{th-homology-iso} and 
    Corollary \ref{cor-intesection-homology} hold. 
  \end{enumerate}
\end{cor}

\section{Examples}
\label{section:examples}
\subsection{The case where $k=n=2$}
We consider the case where 
a $3\times 6$ matrix $z$ defines six lines in $\P^2$.  
The twisted homology group for $z\in Z^{(0)}$ has been extensively studied
(e.g., \cite{Y}). 
We use our results for $z_0 \in Z^{(1)}$ with $|z_0 \la 123\ra |=0$. 
We assume $z\in Z^{(0)}$ and $z_0$ defines the line arrangements 
in Figure \ref{fig:Tz-Tz_0}, 
where the label ``$j$'' denotes the line $(L_j =0)$. 
Recall that $(L_0=0)$ defines the line at infinity. 
\begin{figure}[h]
  \setlength\unitlength{0.8pt}
  \begin{picture}(200,150)(0,-20)
    \put(0,0){\line(1,0){200}} \put(0,5){1}
    \put(0,-10){\line(4,1){200}} \put(0,-25){2}
    \put(70,-15){\line(1,1){130}} \put(60,-25){3}
    \put(130,-15){\line(1,3){50}} \put(120,-25){4}
    \put(170,-15){\line(0,1){150}} \put(175,-20){5}
    \put(50,120){$T_{z}$}
  \end{picture}  \qquad  \quad
  \begin{picture}(200,150)(0,-20)
    \put(0,0){\line(1,0){200}} \put(0,5){1}
    \put(45,-10){\line(4,1){160}} \put(30,-20){2}
    \put(70,-15){\line(1,1){130}} \put(80,-25){3}
    \put(130,-15){\line(1,3){50}} \put(130,-25){4}
    \put(170,-15){\line(0,1){150}} \put(175,-20){5}
    \put(50,120){$T_{z_0}$}
  \end{picture}
  \caption{$T_z$ and $T_{z_0}$ in $\R^2$.}
  \label{fig:Tz-Tz_0}
\end{figure}
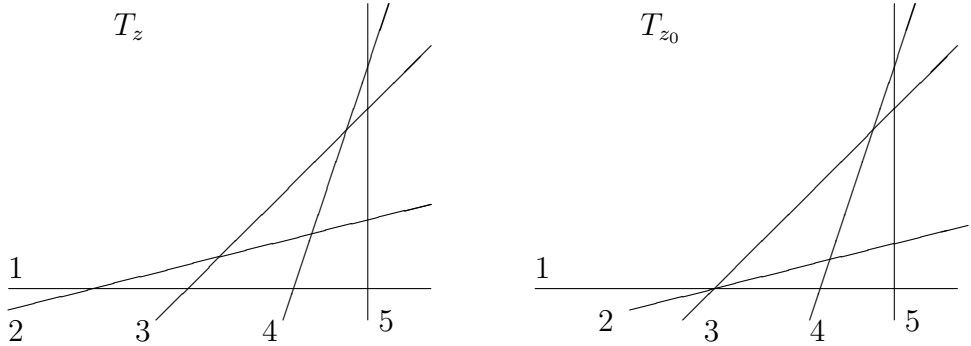

\subsubsection{Cohomology}
In this case, the vanishing form is $\f^{\van}=\f \la 123 \ra$. 
We consider the two intersection numbers 
$\cI_0^c (\f \la 012 \ra ,\f \la 012 \ra)$ and $\cI_0^c (\f \la 012 \ra ,\f \la 234 \ra)$ 
as examples. 
More precisely, we should denote the forms as 
$\fl^c_{\pm} (\f \la 012 \ra)$ and $\fl^c_{-}( \f \la 234 \ra )$; 
however, we use the same notation for simplicity. 
Indeed, we can write
\begin{align*}
  \fl^c_{\pm}(\f\la J\ra ) 
  =\dt\log(L_{j_1}/L_{j_0})\wedge \cdots \wedge \dt\log(L_{j_k}/L_{j_0})
  \left( =\f\la J\ra \in \W^k(T_{z_0})\right)
\end{align*}
as $k$-forms. 
To calculate the intersection numbers by the usual method, we need to consider 
the blow-up of $\P^2 (\supset T_{z_0})$ at the point $P(123)$. 
The exponent corresponding to the exceptional divisor is then
$\a_{123}=\a_1+\a_2+\a_3$, which is not an integer according to the assumption in (\ref{eq:condition-alpha}). 
By calculating the residues at each intersection point (of two lines), 
we obtain the following:
\begin{align}
  \label{eq:cohomology-ex-ans-1}
  \cI_0^c (\f \la 012 \ra ,\f \la 012 \ra) &= (2\pi\sqrt{-1})^2 
  \left( \frac{1}{\a_0 \a_1}+\frac{1}{\a_0 \a_2}+\frac{1}{\a_1 \a_{123}}+\frac{1}{\a_2 \a_{123}}\right),\\ 
  \label{eq:cohomology-ex-ans-2}
  \cI_0^c (\f \la 012 \ra ,\f \la 234 \ra) &= -(2\pi\sqrt{-1})^2 \cdot \frac{1}{\a_2 \a_{123}} .
\end{align}
For further details of this method, refer to \cite{M1}. 
Note that this method requires a slightly complicated calculation. 
For example, the sign ``$-$'' appears in (\ref{eq:cohomology-ex-ans-2}) when the pull-backs of the $2$-forms under the blow-up are expressed in the local coordinate system. 

We now calculate the intersection numbers using our method. 
By Fact \ref{fact:intersection} and $\f^{\van}=\f \la 123 \ra$, 
it is easy to obtain 
intersection numbers for $z\in Z^{(0)}$ as follows: 
{\allowdisplaybreaks
\begin{align*}
  \cI^c (\f^{\van} ,\f^{\van})
  &=(2\pi\sqrt{-1})^2 \cdot \frac{\a_1 +\a_2 +\a_3}{\a_1 \a_2 \a_3}
  =(2\pi\sqrt{-1})^2 \cdot \frac{\a_{123}}{\a_1 \a_2 \a_3}, \\
  \cI^c (\f^{\van} ,\f \la 012 \ra) 
  &=\cI^c (\f \la 012 \ra ,\f^{\van}) 
  =(2\pi\sqrt{-1})^2 \cdot \frac{1}{\a_1 \a_2}, \\   
  \cI^c (\f^{\van} ,\f \la 234 \ra) 
  &=\cI^c (\f \la 234 \ra ,\f^{\van}) 
  =(2\pi\sqrt{-1})^2 \cdot \frac{1}{\a_2 \a_3}, \\ 
  \cI^c (\f \la 012 \ra ,\f \la 012 \ra) 
  &=(2\pi\sqrt{-1})^2 \cdot \frac{\a_0 +\a_1 +\a_2}{\a_0 \a_1 \a_2}, \\ 
  \cI^c (\f \la 012 \ra ,\f \la 234 \ra) 
  &=\cI^c (\f \la 234 \ra ,\f \la 012 \ra) 
  =0 ,
\end{align*}
}\noindent
where we denote $\f^{\van}_{\pm}$ by $\f^{\van}$ for simplicity. 
These intersection numbers are determined by combinatorial data. 
The formula in Corollary \ref{cor-intesection-cohomology} shows that
\begin{align*}
  & \cI^c_0 (\f \la 012 \ra ,\f \la 012 \ra) =\cI^c (\f \la 012 \ra ,\f \la 012 \ra)
  -\frac{\cI^c(\f \la 012 \ra ,\f^{\van}) \cdot \cI^c(\f^{\van},\f \la 012 \ra)}
  {\cI^c(\f^{\van},\f^{\van})} \\
  &=(2\pi\sqrt{-1})^2 \left(
    \frac{\a_0 +\a_1 +\a_2}{\a_0 \a_1 \a_2} 
    -\frac{\frac{1}{\a_1 \a_2} \cdot \frac{1}{\a_1 \a_2}}{\frac{\a_{123}}{\a_1 \a_2 \a_3}}
  \right) \\
  &=(2\pi\sqrt{-1})^2 \cdot 
  \frac{\a_0 \a_1+\a_0 \a_2+\a_1 \a_{123}+\a_2 \a_{123}}{\a_0 \a_1 \a_2 \a_{123}}
\end{align*}
(recall that $\a_{123}=\a_1 +\a_2 +\a_3$), and 
\begin{align*}
  & \cI^c_0 (\f \la 012 \ra ,\f \la 234 \ra) =\cI^c (\f \la 012 \ra ,\f \la 234 \ra)
  -\frac{\cI^c(\f \la 012 \ra ,\f^{\van}) \cdot \cI^c(\f^{\van},\f \la 234 \ra)}
  {\cI^c(\f^{\van},\f^{\van})} \\
  &=0-(2\pi\sqrt{-1})^2 \cdot 
  \frac{\frac{1}{\a_1 \a_2} \cdot \frac{1}{\a_2 \a_3}}{\frac{\a_{123}}{\a_1 \a_2 \a_3}} 
  =-(2\pi\sqrt{-1})^2 \cdot \frac{1}{\a_2 \a_{123}} ,
\end{align*}
which coincide with (\ref{eq:cohomology-ex-ans-1}) and (\ref{eq:cohomology-ex-ans-2}), respectively. 
We can obtain these intersection numbers through simple calculations.

\subsubsection{Homology}
As the picture of $T_z$ is the same as that in \cite[p.187]{Y}, 
we use the branch of $u_z$ introduced there. 
The vanishing chamber $\De^{\van}\subset T_z\cap \R^2$ is the triangle 
surrounded by $(L_1=0)$, $(L_2=0)$, and $(L_3=0)$. 
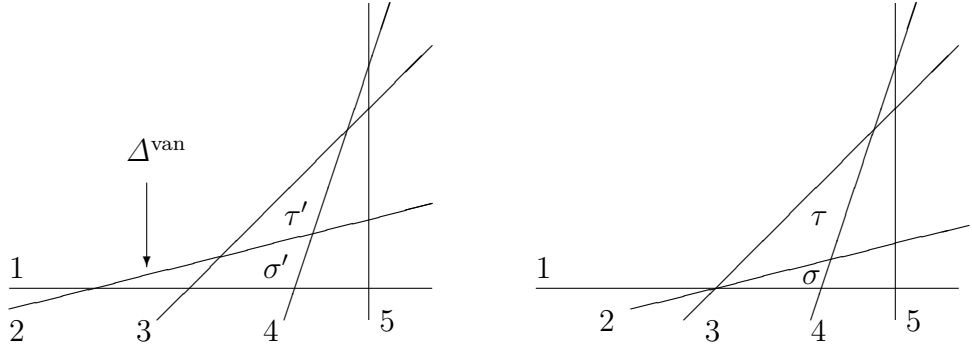
\begin{figure}[h]
  \setlength\unitlength{0.8pt}
  \begin{picture}(200,150)(0,-20)
    \put(0,0){\line(1,0){200}} \put(0,5){1}
    \put(0,-10){\line(4,1){200}} \put(0,-25){2}
    \put(70,-15){\line(1,1){130}} \put(60,-25){3}
    \put(130,-15){\line(1,3){50}} \put(120,-25){4}
    \put(170,-15){\line(0,1){150}} \put(175,-20){5}
    \put(65,50){\vector(0,-1){40}}
    \put(55,60){$\De^{\van}$}
    \put(130,30){$\tau'$}
    \put(120,5){$\si'$}
  \end{picture}
  \qquad \quad 
  \begin{picture}(200,150)(0,-20)
    \put(0,0){\line(1,0){200}} \put(0,5){1}
    \put(45,-10){\line(4,1){160}} \put(30,-20){2}
    \put(70,-15){\line(1,1){130}} \put(80,-25){3}
    \put(130,-15){\line(1,3){50}} \put(130,-25){4}
    \put(170,-15){\line(0,1){150}} \put(175,-20){5}
    \put(130,30){$\tau$}
    \put(126,2){$\si$}
  \end{picture}
  \caption{Chambers.}
  \label{fig:chambers}
\end{figure}

We consider two chambers $\si, \tau \subset T_{z_0}\cap \R^2$ 
(see the right side of Figure \ref{fig:chambers}): 
\begin{itemize}
\item $\si$ is the triangle surrounded by 
  $(L_1=0)$, $(L_2=0)$, and $(L_4=0)$, 
\item $\tau$ is the triangle surrounded by 
  $(L_2=0)$, $(L_3=0)$, and $(L_4=0)$. 
\end{itemize}
We denote the twisted cycles defined by these chambers 
as $\si_{\pm}, \tau_{\pm} \in H_k (T_{z_0} ,u_{z_0}^{\pm 1})$. 
As examples, we consider the two intersection numbers 
$\cI_0^h (\si_{+} ,\si_{-})$ and $\cI_0^h (\si_{+} ,\tau_{-})$. 
According to \cite[II-Section 3]{KY}, 
the intersection numbers are calculated as 
\begin{align}
  \label{eq:homology-ex-ans-1}
  \cI_0^h (\si_{+} ,\si_{-}) &= \frac{\l_1\l_2-1}{(\l_1-1)(\l_2-1)}
  \cdot \frac{\l_1\l_2\l_3\l_4-1}{(\l_4-1)(\l_1\l_2\l_3-1)} ,\\ 
  \label{eq:homology-ex-ans-2}
  \cI_0^h (\si_{+} ,\tau_{-}) &= -\frac{1}{\l_2-1} 
  \left( 1+\frac{1}{\l_4-1}+\frac{1}{\l_1\l_2\l_3-1} \right) \\
  \nonumber
  &=\frac{-(\l_1\l_2\l_3\l_4-1)}{(\l_2-1)(\l_4-1)(\l_1\l_2\l_3-1)} .
\end{align}
% To calculate them, 
For the calculation in \cite{KY}, 
we need to consider 
the blow-up of $\P^2 (\supset T_{z_0})$ at the point $P(123)$. 
This makes it difficult to investigate the intersections between 
the blown-up chambers. 

We now calculate the intersection numbers using our method. 
It is easy to see that $\si_{\pm}$ and $\tau_{\pm}$ are expressed as 
\begin{align*}
  \si_{\pm} =\fl^c_{\pm}(\si'_{\pm}) ,\qquad
  \tau_{\pm} =\fl^c_{\pm}(\tau'_{\pm})
\end{align*}
for certain $\si'_{\pm} ,\tau'_{\pm} \in H_k(T_{z},u_z^{\pm1})$;
see the left side of Figure \ref{fig:chambers} for $\si'$ and $\tau'$.
By \cite[VIII.3.4]{Y}, we have
{\allowdisplaybreaks
\begin{align*}
  \cI^h (\De^{\van}_{+} ,\De^{\van}_{-})&=\frac{\l_1\l_2\l_3-1}{(\l_1-1)(\l_2-1)(\l_3-1)}, \\
  \cI^h (\De^{\van}_{+} ,\si'_{-})&=\frac{-(\l_1\l_2-1)}{(\l_1-1)(\l_2-1)(\l_3-1)}, \\ 
  \cI^h (\De^{\van}_{+} ,\tau'_{-})&=\frac{1}{(\l_2-1)(\l_3-1)}, \\   
  \cI^h (\si'_{+} ,\De^{\van}_{-})&=\frac{-\l_3(\l_1\l_2-1)}{(\l_1-1)(\l_2-1)(\l_3-1)}, \\ 
  \cI^h (\si'_{+} ,\si'_{-})&=\frac{(\l_1\l_2-1)(\l_3\l_4-1)}{(\l_1-1)(\l_2-1)(\l_3-1)(\l_4-1)} , \\
  \cI^h (\si'_{+} ,\tau'_{-} )&=\frac{-(\l_3\l_4-1)}{(\l_2-1)(\l_3-1)(\l_4-1)} .
\end{align*}
}\noindent
These are easier to calculate than (\ref{eq:homology-ex-ans-1}) and (\ref{eq:homology-ex-ans-2}), 
because we can check the intersections of chambers in $T_z \cap \R^k$ 
directly (without blowing-up). 
The formula in Corollary \ref{cor-intesection-homology} 
(which is valid by Corollary \ref{cor-intesection-homology-2}) shows that
\begin{align*}
  &\cI^h_0 (\si_{+}, \si_{-}) =\cI^h (\si'_{+} ,\si'_{-})
  -\frac{\cI^h(\si'_{+}, \De^{\van}_{-}) \cdot \cI^h(\De^{\van}_{+},\si'_{-})}
  {\cI^h(\De^{\van}_{+},\De^{\van}_{-})} \\
  &=\frac{(\l_1\l_2-1)(\l_3\l_4-1)}{(\l_1-1)(\l_2-1)(\l_3-1)(\l_4-1)}
  -\frac{\frac{-\l_3(\l_1\l_2-1)}{(\l_1-1)(\l_2-1)(\l_3-1)}
    \cdot \frac{-(\l_1\l_2-1)}{(\l_1-1)(\l_2-1)(\l_3-1)}}
  {\frac{\l_1\l_2\l_3-1}{(\l_1-1)(\l_2-1)(\l_3-1)}} \\
  &=\frac{(\l_1\l_2-1)(\l_1\l_2\l_3\l_4-1)}{(\l_1-1)(\l_2-1)(\l_4-1)(\l_1\l_2\l_3-1)},
\end{align*}
and
\begin{align*}
  &\cI^h_0 (\si_{+}, \tau) =\cI^h (\si'_{+} ,\tau'_{-})
  -\frac{\cI^h(\si'_{+}, \De^{\van}_{-}) \cdot \cI^h(\De^{\van}_{+},\tau'_{-})}
  {\cI^h(\De^{\van},\De^{\van}_{-})} \\
  &=\frac{-(\l_3\l_4-1)}{(\l_2-1)(\l_3-1)(\l_4-1)}
  -\frac{\frac{-\l_3(\l_1\l_2-1)}{(\l_1-1)(\l_2-1)(\l_3-1)}
    \cdot \frac{1}{(\l_2-1)(\l_3-1)}}{\frac{\l_1\l_2\l_3-1}{(\l_1-1)(\l_2-1)(\l_3-1)}} \\
  &=\frac{-(\l_1\l_2\l_3\l_4-1)}{(\l_2-1)(\l_4-1)(\l_1\l_2\l_3-1)} .
\end{align*}
In this way, we can calculate the intersection numbers for $z_0\in Z^{(1)}$ 
from those for $z\in Z^{(0)}$.

\subsection{Another basis of twisted cohomology}
Using the intersection numbers, 
we derive another basis of the twisted cohomology group for $z \in Z^{(1)}$.
Along with those in Fact \ref{fact:cohomology-basis-Z1}, 
this basis is useful when we consider the contiguity relations, 
as discussed in Section \ref{section:contiguity}. 
As in Section \ref{section:vanishing}, 
let $z\in Z^{(0)}$, $z_0 \in Z^{(1)}$ and $|z_0 \la J^{\van}\ra| =0$. 
% Note that we can write
% \begin{align*}
%   \fl^c_{\pm}(\f\la J\ra ) 
%   =\dt\log(L_{j_1}/L_{j_0})\wedge \cdots \wedge \dt\log(L_{j_k}/L_{j_0})
%   \left( =\f\la J\ra \in \W^k(T_{z_0})\right)
% \end{align*}
% as $k$-forms. 
\begin{proposition}
  \label{prop:cohomology-basis-Z1-2}
  Let $p\not\in J^{\van}$ and $l,l'\in J^{\van}$ with $l\neq l'$. 
  The set 
  \begin{align*}
    \{ \fl^c_{\pm}(\f\la J\ra ) \mid J\in \under{l'}{\cJ}{l}-\{ \under{l'}{J^{\van}}{p}\} \}
  \end{align*}
  then gives bases of $H^k(\W^{\bu}(T_{z_0}),\na^{\pm \a}_0)$. 
\end{proposition}
\begin{proof}
  It suffices for the intersection matrix 
  \begin{align*}
    C_0 = \Big( \cI^c_0 (\fl^c_{+}(\f\la J\ra ),\fl^c_{-}(\f\la J'\ra)) \Big)_{
      J\in \under{l'}{\cJ}{l}-\{ \under{l'}{J^{\van}}{p}\} ,\ 
      J'\in \under{l}{\cJ}{l'}-\{ \under{l}{J^{\van}}{p}\}
    }
  \end{align*}
  to be invertible. 
  We align $\under{l'}{\cJ}{l}-\{ \under{l'}{J^{\van}}{p}\}$ according to
  an arbitrary order, and define an order on $\under{l}{\cJ}{l'}-\{ \under{l}{J^{\van}}{p}\}$
  by the correspondence 
  \begin{align*}
    \under{l'}{\cJ}{l}-\{ \under{l'}{J^{\van}}{p}\} \ni 
    J\leftrightarrow \under{l}{J}{l'}
    \in \under{l}{\cJ}{l'}-\{ \under{l}{J^{\van}}{p}\} .
  \end{align*}
  We consider the systems
  $\{ J^{\van} \}\cup (\under{l'}{\cJ}{l}-\{ \under{l'}{J^{\van}}{p}\})$, 
  $\{ J^{\van} \}\cup (\under{l}{\cJ}{l'}-\{ \under{l}{J^{\van}}{p}\})$ and 
  their intersection matrix
  \begin{align*}
    C&=\Big( \cI^c (\f\la J\ra ,\f\la J'\ra) \Big)_{
      J \in \{ J^{\van} \}\cup (\under{l'}{\cJ}{l}-\{ \under{l'}{J^{\van}}{p}\}) ,\ 
      J'\in \{ J^{\van} \}\cup (\under{l}{\cJ}{l'}-\{ \under{l}{J^{\van}}{p}\})
    } \\
    &=
    \begin{pmatrix}
      \cI^c (\f^{\van}_{+} ,\f^{\van}_{-}) &\cdots &
      \cI^c (\f^{\van}_{+} ,\f\la J'\ra) &\cdots  \\
      \vdots &\ddots & \vdots & \\
      \cI^c (\f\la J\ra ,\f^{\van}_{-}) &\cdots &
      \cI^c (\f\la J\ra ,\f\la J'\ra) & \\ 
      \vdots & & \vdots & \ddots
    \end{pmatrix}
  \end{align*}
  on $H^k(\W^{\bu}(T_{z}),\na^{\pm \a})$.
  Using Fact \ref{fact:intersection}, 
  $C$ can be written as 
  \begin{align*}
    C=
    \begin{pmatrix}
      \cI^c (\f^{\van}_{+} ,\f^{\van}_{-}) &\cdots &
      \cI^c (\f^{\van}_{+} ,\f\la J'\ra) &\cdots  \\
      \vdots & &  & \\
      \cI^c (\f\la J\ra ,\f^{\van}_{-}) & &
      C_1 & \\ 
      \vdots & &  & 
    \end{pmatrix}, 
  \end{align*}
  where $C_1$ is an invertible diagonal matrix of size $\binom{k+n}{k}-1$. 
  To prove $|C_0|\neq 0$, we compute the determinant $|C|$ in two ways. 
  First, we add 
  \begin{align*}
    -\frac{\cI^c (\f\la J\ra ,\f^{\van}_{-})}{\cI^c (\f^{\van}_{+} ,\f^{\van}_{-})}
    \cdot (\textrm{the first row})
  \end{align*}
  to the $J$-th row of $C$. 
  The $(J,J')$-th entry of $C_1$ becomes
  \begin{align*}
    \cI^c (\f\la J\ra ,\f\la J'\ra) 
    -\frac{\cI^c (\f\la J\ra ,\f^{\van}_{-})}{\cI^c (\f^{\van}_{+} ,\f^{\van}_{-})} 
    \cdot \cI^c (\f^{\van}_{+} ,\f\la J'\ra)
    =\cI^c_0 (\fl^c_{+}(\f\la J\ra ),\fl^c_{-}(\f\la J'\ra)).
  \end{align*}
  We then obtain
  \begin{align*}
    |C|=
    \begin{vmatrix}
      \cI^c (\f^{\van}_{+} ,\f^{\van}_{-}) &\cdots &
      \cI^c (\f^{\van}_{+} ,\f\la J'\ra) &\cdots  \\
      0 & &  & \\
      \vdots & &
      C_0 & \\ 
      0 & &  & 
    \end{vmatrix}
    =\cI^c (\f^{\van}_{+} ,\f^{\van}_{-}) \cdot |C_0| ,
  \end{align*}
  and hence our claim is equivalent to $|C|\neq 0$ by (\ref{eq:self-intersection-van-cohomology}). 
  Next, we compute $|C|$ in another way. 
  We add each
  \begin{align*}
    -\frac{\cI^c (\f^{\van}_{+},\f\la \under{l}{J}{l'}\ra)}{\cI^c (\f\la J\ra ,\f\la \under{l}{J}{l'}\ra)}
    \cdot (\textrm{the $J$-th row})
  \end{align*}
  to the first row of $C$ (note that the denominator is the diagonal entry of $C_1$). 
  We set 
  \begin{align*}
    c_2 =\cI^c (\f^{\van}_{+} ,\f^{\van}_{-}) 
    -\sum_{J\in \under{l'}{\cJ}{l}-\{ \under{l'}{J^{\van}}{p}\}}
    \frac{\cI^c (\f^{\van}_{+},\f\la \under{l}{J}{l'}\ra)}{\cI^c (\f\la J\ra ,\f\la \under{l}{J}{l'}\ra)}
    \cdot \cI^c (\f\la J\ra ,\f^{\van}_{-}) 
  \end{align*}
  such that we have 
  \begin{align*}
    |C| =
    \begin{vmatrix}
      c_2 &0 &
      \cdots & 0  \\
      \vdots & &  & \\
      \cI^c (\f\la J\ra ,\f^{\van}_{-}) & &
      C_1 & \\ 
      \vdots & &  & 
    \end{vmatrix}
    =c_2 \cdot |C_1| ,
  \end{align*}
  and our claim is reduced to $c_2 \neq 0$. 
  For $J\in \under{l'}{\cJ}{l}-\{ \under{l'}{J^{\van}}{p}\}$, we have
  \begin{align*}
    \cI^c (\f\la J\ra ,\f^{\van}_{-}) \neq 0 
    \Longleftrightarrow 
    J=\under{l'}{J^{\van}}{q} \ (q\not\in J^{\van},\ q\neq p) .
  \end{align*}
  Since $\under{l}{(\under{l'}{J^{\van}}{q})}{l'}=\under{l}{J^{\van}}{q}$ 
  as sets, we obtain 
  \begin{align*}
    &\cI^c (\f^{\van}_{+},\f\la \under{l}{(\under{l'}{J^{\van}}{q})}{l'}\ra) 
    \cdot \cI^c (\f\la \under{l'}{J^{\van}}{q}\ra ,\f^{\van}_{-})
    =-\frac{(\tpi)^k \a_{l}}{\prod_{j\in J^{\van}}\a_j} \cdot 
      \frac{(\tpi)^k\a_{l'}}{\prod_{j\in J^{\van}}\a_j} , \\
    & \cI^c (\f\la \under{l'}{J^{\van}}{q}\ra ,\f\la \under{l}{(\under{l'}{J^{\van}}{q})}{l'}\ra)
    =\frac{(\tpi)^k\a_l \a_{l'}}{\a_q \prod_{j\in J^{\van}} \a_j} .
  \end{align*}
  Thus, we have
  \begin{align*}
    &\frac{c_2}{(\tpi)^k} 
    = \frac{\sum_{j\in J^{\van}}\a_j }{\prod_{j\in J^{\van}}\a_j }
    +\sum_{\substack{q\not\in J^{\van}\\ q\neq p}}
    \frac{\a_{l}}{\prod_{j\in J^{\van}}\a_j} \cdot \frac{\a_{l'}}{\prod_{j\in J^{\van}}\a_j}
    \cdot \frac{\a_q \prod_{j\in J^{\van}} \a_j}{\a_l \a_{l'}} \\
    &=\frac{1}{\prod_{j\in J^{\van}}\a_j }\Bigg(
      \sum_{j\in J^{\van}}\a_j +\sum_{\substack{q\not\in J^{\van}\\ q\neq p}} \a_q
    \Bigg)
    =\frac{\sum_{j\neq p}\a_j}{\prod_{j\in J^{\van}}\a_j}
    =\frac{-\a_p}{\prod_{j\in J^{\van}}\a_j} \neq 0 ,
  \end{align*}
  and the proof is complete. 
\end{proof}

\section{An application --- contiguity relations}
\label{section:contiguity}
As shown in \cite{GM-PC}, the intersection numbers on twisted cohomology groups are 
useful for providing explicit expressions of the contiguity relations for $z\in Z^{(0)}$. 
In this section, we show that similar discussions remain valid when $z\in Z^{(1)}$.

\subsection{Contiguity relations for hypergeometric integrals}
\label{subsection:contiguity-integral}
We use the same notation as in Section \ref{section:vanishing}. 
% By using a suitable coordinates change if necessary, we may assume that $0\in J^{\van}$. 
We fix $j_0\in J^{\van}$, $q\not\in J^{\van}$ and a twisted cycle $\si \in H_k(T_{z_0},u_{z_0})$. 
We set 
$\cJ^{\circ}=\under{q}{\cJ}{j_0} -\{ J^{\van}\}$ and
take the basis of $H^k(\W^{\bu}(T_{z_0}),\na^{\a}_0)$ as 
$\{ \f\la J\ra \mid J\in \cJ^{\circ} \}$ (cf. Fact \ref{fact:cohomology-basis-Z1}).
We consider the following column vector consisting of hypergeometric integrals:
\begin{align*}
  \bF(\a;z_0 )={\TP{\left( \dots, \int_{\si} u_{z_0}(t)\f\la J\ra ,\dots \right)}}_{J\in \cJ^{\circ}} .
\end{align*}
% By varying $z_0$ preserving the condition $|z_0\la J^{\van}\ra | =0$, 
% we can regard $F(\a;z_0)$ as a function in 
% $z_0 \in \{ z\in Z^{(1)} \mid ~|z\la J^{\van}\ra | =0 \}\subset Z^{(1)}$. 

% As in \cite[Section 5]{GM-PC}, 
We fix $l\neq j_0$, 
and set 
\begin{align*}
  &\a^{(l)}=\a +e_l -e_{j_0},  \\
  &\w_0^{(l)}=\w_0 +\dt \log L_l -\dt \log L_{j_0},\quad
    \na_0^{\a^{(l)}} =\dt +\w_0^{(l)}\we  ,
\end{align*}
where $e_l$ is the $l$-th unit vector. 
Note that $\dt \log L_{0} =\dt \log (1)=0$. 
\begin{notation}
  We write 
  \begin{align*}
    &\cohom =H^k (\Omd (T_{z_0}),\na_0^{\a}) ,& 
    &\cohomi =H^k (\Omd (T_{z_0}),\na_0^{\a^{(l)}}) ,\\
    &\cohom^{\vee}=H^k (\Omd (T_{z_0}),\na_0^{-\a}) ,& 
    &\cohomi^{\vee}=H^k (\Omd (T_{z_0}),\na_0^{-\a^{(l)}}) ,&
  \end{align*}
  and denote
  the intersection pairing between 
  $\cohom$ and $\cohom^{\vee}$
  (resp. $\cohomi$ and $\cohomi^{\vee}$)
  as $\cI_0$ (resp. $\cI_0^{(l)}$), for simplicity. 
  For a given $\psi \in \W^k (T_{z_0})$, 
  to clarify which cohomology group $\psi$ belongs to, 
  we write $[\psi]$, $[\psi]_l$, $[\psi]^{\vee}$, and $[\psi]_l^{\vee}$
  to denote the element of $\cohom$, $\cohomi$, $\cohom^{\vee}$, and $\cohomi^{\vee}$, respectively, 
 that is represented by $\psi$. 
  If there is no possibility of confusion, we denote 
  $\fl^c_{\pm}(\f\la J\ra )$ by $\f\la J\ra$, 
  as they are the same $k$-form formally. 
\end{notation}
The following is proved in a similar manner to \cite[Proposition 5.2]{GM-PC}. 
\begin{lemma}[\cite{GM-PC}]
  \label{lem:linear-map}
  The map 
  \begin{align*}
    \Conti{l} : \cohomi \ni [\f ]_l \mapsto 
    \left[ \frac{L_l}{L_{j_0}} \cdot \f \right] \in \cohom 
  \end{align*}
  is a well-defined linear map. 
\end{lemma}
Let $\conti{l} (\a;z_0)$ be the representation matrix of $\Conti{l}$ 
with respect to the bases 
$\{ [\f\la J\ra ]_l \mid J\in \cJ^{\circ} \}$ of $\cohomi$ 
and $\{ [\f\la J\ra ] \mid J\in \cJ^{\circ} \}$ of $\cohom$.
As
\begin{align*}
  \int_{\si} u_{z_0}(t) \cdot \Conti{l} ([\f \la J \ra ]_l)
  =\int_{\si} \left( u_{z_0}(t) \cdot \frac{L_l}{L_{j_0}} \right) \cdot \f\la J\ra
  =\left. \int_{\si} u_{z_0}(t) \f\la J\ra \right|_{\a \to \a^{(l)}} , 
\end{align*}
we have the contiguity relation
\begin{align*}
  \bF (\a^{(l)};z_0)=\conti{l} (\a;z_0) \cdot \bF (\a;z_0) .
\end{align*}
By deriving an expression for $\conti{l} (\a;x)$, 
we obtain an explicit form of the contiguity relation. 
\begin{theorem}
  \label{th:contiguity}
  We define $\under{l}{\cJ^{\circ}}{j_0}$ by replacing $l$ with $q$ in $\cJ^{\circ}$. 
  We next define $\under{j_0}{\cJ^{\circ}}{l}$
  by replacing $j_0$ with $l$ in $\under{l}{\cJ^{\circ}}{j_0}$ 
  (we align their elements according to these correspondences). 
  Namely, 
  \begin{enumerate}[(i)]
  \item if $l\not\in J^{\van}$, we have 
    $\under{l}{\cJ^{\circ}}{j_0} =\under{l}{\cJ}{j_0} -\{ J^{\van}\}$ and 
    $\under{j_0}{\cJ^{\circ}}{l} =\under{j_0}{\cJ}{l} -\{ \under{j_0}{J^{\van}}{l}\}$ 
    (note that if $l=q$, then $\under{l}{\cJ^{\circ}}{j_0} =\cJ^{\circ}$), 
  \item if $l\in J^{\van}$, we have 
    $\under{l}{\cJ^{\circ}}{j_0} =\under{l}{\cJ}{j_0} -\{ \under{l}{J^{\van}}{q} \}$ and 
    $\under{j_0}{\cJ^{\circ}}{l} =\under{j_0}{\cJ}{l} -\{ \under{j_0}{(\under{l}{J^{\van}}{q})}{l}\}$. 
  \end{enumerate}
  The representation matrix $\conti{l}(\a;z_0)$ admits the expression  
  \begin{align*}
    \conti{l} (\a;z_0)=C(\a^{(l)})P_l (\a^{(l)})^{-1} D_l(z_0) Q_l (\a) C(\a)^{-1}, 
  \end{align*}
  where 
  \begin{align*}
    &D_l(z_0)=\diag \left( \ldots,\frac{|z_0 \la \under{j_0}{J}{l} \ra |}{|z_0 \la J \ra |} 
      ,\ldots \right)_{J\in\under{l}{\cJ^{\circ}}{j_0}} ,&
    &C(\a)=\Bigl( \cI_0 (\f \la I \ra , \f \la J \ra )\Bigr)_{I,J\in \cJ^{\circ}} , \\
    &P_l(\a)=\Bigl( \cI_0 (\f \la I \ra , \f \la J \ra ) 
    \Bigr)_{I\in \under{j_0}{\cJ^{\circ}}{l},J\in \cJ^{\circ}} ,& 
    &Q_l(\a)=\Bigl( \cI_0 (\f \la I \ra , \f \la J \ra ) 
    \Bigr)_{I\in \under{l}{\cJ^{\circ}}{j_0},J\in \cJ^{\circ}} .
  \end{align*}
\end{theorem}
The proof is similar to that for \cite[Theorem 5.3]{GM-PC}.
Note that we can evaluate each entry of the above intersection matrices using 
Corollary \ref{cor-intesection-cohomology}. 
\begin{proof}
  By Fact \ref{fact:cohomology-basis-Z1} and Proposition \ref{prop:cohomology-basis-Z1-2}, 
  $\{ [\f\la J\ra ]_l \mid J\in \under{j_0}{\cJ^{\circ}}{l} \}$ and 
  $\{ [\f\la J\ra ] \mid J\in \under{l}{\cJ^{\circ}}{j_0} \}$ are bases of 
  $\cohomi$ and $\cohom$, respectively. 
  By the definition, each element in $\under{j_0}{\cJ^{\circ}}{l}$ is 
  uniquely expressed as $\under{j_0}{J}{l}$ with $J\in \under{l}{\cJ^{\circ}}{j_0}$.  
  By (\ref{eq:phi}), we have 
  \begin{align*}
    \f \la \under{j_0}{J}{l} \ra 
    =\frac{|z_0 \la \under{j_0}{J}{l}\ra|}{\prod_{j\in \under{j_0}{J}{l}} L_j} \dd t
    =\frac{L_{j_0}}{L_l} \cdot 
    \frac{|z_0 \la \under{j_0}{J}{l}\ra|}{\prod_{j\in J} L_j} \dd t
    =\frac{L_{j_0}}{L_l} \cdot 
    \frac{|z_0 \la \under{j_0}{J}{l}\ra|}{|z_0 \la J \ra|} \cdot 
    \f \la J \ra . 
  \end{align*}
  The representation matrix of $\Conti{l}$ 
  with respect to these bases then coincides with $D_l(z_0)$. 
  By
  \begin{align*}
    \cI_0^{(l)} ([\f \la J \ra ]_l, [\f \la J' \ra ]_l^{\vee} )
    =\cI_0 ([\f \la J \ra ], [\f \la J' \ra ]^{\vee} ) \Big|_{\a \to \a^{(l)}}
  \end{align*}
  and the linearity of the intersection forms $\cI_0$ and $\cI_0^{(l)}$, 
  we have  
  \begin{align*}
    \begin{pmatrix}
      \vdots \\ [\f \la {J'} \ra ] \\ \vdots
    \end{pmatrix}_{J'\in \under{l}{\cJ^{\circ}}{j_0}}
    &=Q_l(\a)C(\a)^{-1}
    \begin{pmatrix}
      \vdots \\ [\f \la {J} \ra ] \\ \vdots
    \end{pmatrix}_{J\in \cJ^{\circ}} ,\\ 
    \begin{pmatrix}
      \vdots \\ [\f \la {J'} \ra ]_l \\ \vdots
    \end{pmatrix}_{J'\in \under{j_0}{\cJ^{\circ}}{l}}
    &=P_l(\a^{(l)})C(\a^{(l)})^{-1}
    \begin{pmatrix}
      \vdots \\ [\f \la {J} \ra ]_l \\ \vdots
    \end{pmatrix}_{J\in \cJ^{\circ}} .
  \end{align*}
  Thus, the proof is complete. 
\end{proof}

\subsection{Contiguity relations for hypergeometric series}
As in \cite{GM-PC} and \cite{TGKT}, 
the contiguity relations are utilized for algebraic statistics. 
The hypergeometric polynomial is regarded as 
the normalizing constant of 
the hypergeometric distribution of 
two-way contingency tables 
with fixed marginal sums. 
Using our result for $z_0 \in Z^{(1)}$, 
we can evaluate the normalizing constants when 
a fixed cell is zero. 
To demonstrate such an application, we derive the contiguity relations for 
hypergeometric series.

To relate hypergeometric integrals to series, 
we use a specified $z\in Z^{(01)}$. 
Let  $x_{ij}$ $(1\le i\le k,\ 1\le j\le n)$ be 
$k\times n$ variables and $x=(x_{ij})$ be the matrix arranging them.
We set a $(k+1)\times (k+n+2)$ matrix 
\begin{align*}
  \tx 
  &=(\tx_{ij})_{
    \begin{subarray}{l}
      0\le i\le k\\ 0\le j\le k+n+1
    \end{subarray}} \\
  &=\bordermatrix{
  &0&1& \cdots &k&k+1& \cdots &k+n& k+n+1\cr
  0&1&0& \cdots      &0&1  & \cdots &1 &1 \cr
  1&0&1&   &0&x_{11}&\cdots&x_{1n} &1 \cr
  \vdots&\vdots & &\ddots &\vdots &\vdots &\ddots&\vdots &\vdots \cr
  k&0&0&\cdots  &1&x_{k1}&\cdots&x_{kn} &1 \cr
  } .
\end{align*}
We set 
\begin{align*}
  X^{(i)}=\{ x\in M(k,n;\C )\mid \tx \in Z^{(i)} \} ,\quad i=0,1,
\end{align*}
and 
\begin{align*}
  X_0^{(1)}=\{ x=(x_{ij})\in X^{(1)} \mid \textrm{exactly one $x_{\imath \jmath}$ is zero} \} .
\end{align*}
Note that, for $\tx \in Z^{(01)}$, 
the linear forms of the $L_j$ are written as 
\begin{align*}
  &L_j=t_j \ (0\le j\le k),\quad 
  L_{k+j}=t_0+t_1x_{1j}+\cdots+t_kx_{kj} \ (1\le j\le n), \\
  &L_{k+n+1}=t_0+t_1+\cdots+t_k .
\end{align*}
% We put 
% \begin{align*}
%   M'(k,n;\Z_{\geq 0}) =\{ m=(m_{ij})\in M(k,n;\Z_{\geq 0}) \mid m_{i_0 j_0} =0 \}
% \end{align*}
By setting $x_{\imath \jmath}=0$ in \cite[Proposition 6.1]{GM-PC} if necessary, 
we obtain a power series expansion of a hypergeometric integral. 
\begin{lemma}[\cite{GM-PC}]
  We assume that $x\in X^{(0)} \cup X_0^{(1)}$ and 
  each entry of $x$ is a real number that
  is sufficiently close to $0$. 
  The standard simplex 
  \begin{align*}
    \{(t_1,\dots,t_k)\in \R^k\mid t_1,\dots,t_k<0, \ t_1+\cdots+t_k>-1\}
  \end{align*}
  belongs to $\bC (\cA_{\tx})$. 
  Let $\triangle$ be a twisted cycle defined by  
  the simplex loading $u_{\tx}(t)$ whose branch is given
  by the following. 
  \begin{align*}
    \begin{array}{|c|c|c|c|}
      \hline
      &i=1,\ldots ,k & i=k+1,\ldots ,k+n & i=k+n+1 \\ \hline
      \arg L_i &-\pi & 0 & 0 \\ \hline
    \end{array}
  \end{align*}
  Then, the hypergeometric integral
  \begin{align*}
    F (\a,x) 
    &=\int_{\triangle} u_{\tx}(t) \f \la 01\cdots k\ra \\
    &=\int_{\triangle}\prod_{i=1}^k t_i^{\a_i} 
      \cdot \prod_{j=1}^n \Bigl( 1+\sum_{i=1}^k x_{ij}t_i \Bigr)^{\a_{k+j}} 
      \cdot \Bigl( 1+\sum_{i=1}^k t_i \Bigr)^{\a_{k+n+1}} 
      \frac{\dd t}{t_1 \cdots t_k}
  \end{align*}
  admits the series expansion
  \begin{align*}
    e^{-\pi \sqrt{-1}(\a_1 +\cdots +\a_k)} \cdot 
    \prod_{i=1}^k \G(\a_i)\G(-\a_i+1) \cdot 
    \prod_{j=1}^{n+1}\G(\a_{k+j} +1) \cdot S(\a ;x) ,
  \end{align*}
  where 
  \begin{align*}
    S(\a;x)=
    \sum_{m=(m_{ij})\in M(k,n;\Z_{\geq 0})}
    \frac{1}{\G_m(\a)} \cdot \prod_{i,j}x_{ij}^{m_{ij}} 
  \end{align*}
  and 
  \begin{align*}
    \G_m(\a)
    =&\prod_{i=1}^k \G(-\a_i -\sum_{j=1}^n m_{ij} +1)
       \cdot \prod_{j=1}^n \G(\a_{k+j}-\sum_{i=1}^k m_{ij}+1) \\
     &\cdot \G(\sum_{i=1}^{k}\a_i+\a_{k+n+1}+\sum_{i=1}^k \sum_{j=1}^n m_{ij}+1)
       \cdot \prod_{i=1}^k \prod_{j=1}^n \G(m_{ij}+1).
  \end{align*}
  Here, if $x\in X_0^{(1)}$ with $x_{\imath \jmath}=0$, 
  we regard $(x^0_{\imath \jmath})^{m_{\imath \jmath}}$ as 
  \begin{align*}
    (x^0_{\imath \jmath})^{m_{\imath \jmath}} =
    \begin{cases}
      1 & (\textrm{if} \quad m_{\imath \jmath} =0), \\
      0 & (\textrm{if} \quad m_{\imath \jmath} \neq 0) .
    \end{cases}
  \end{align*}
\end{lemma}
We say that $S(\a ;x)$ is the \textit{hypergeometric series}. 
Hereafter, we consider 
\begin{align*}
  x_0 =(x^0_{ij}) \in X_0^{(1)} ,\quad x^0_{\imath \jmath} =0 
\end{align*}
for an application. In this case, we have 
\begin{align*}
  J^{\van}
  &=\{ 0,1,2,\ldots ,k ,k+\jmath \} -\{ \imath \} \\
  &=\la 0,1,2,\ldots ,\imath -1 ,k+\jmath ,\imath +1 ,\ldots ,k \ra , 
\end{align*}
because $|\tx_0 \la 0,1,2,\ldots ,\imath -1,k+\jmath ,\imath +1 ,\ldots ,k \ra | =x^0_{\imath \jmath}$. 

We take a basis of twisted cohomology groups as 
$\{ \f\la J\ra \mid J\in \cJ^{\circ} \}$, 
where $\cJ^{\circ}=\under{k+n+1}{\cJ}{0} -\{ J^{\van}\}$ 
(we set $j_0=0\in J^{\van}$ and $q=k+n+1 \not\in J^{\van}$ in Subsection \ref{subsection:contiguity-integral}). 
Replacing $\si$ with $\triangle$, we consider the vector
\begin{align*}
  \bF(\a;\tx_0 )={\TP{\left( \dots, \int_{\triangle} u_{z_0}(t)\f\la J\ra ,\dots \right)}}_{J\in \cJ^{\circ}} .
\end{align*}
The following is obtained in a similar manner to that explained in \cite{GM-PC}. 
\begin{lemma}[\cite{GM-PC}]
  \label{lem:bold-S}
  We set 
  \begin{align*}
    \bS(\a;x_0)
    =e^{\pi \sqrt{-1}(\a_1 +\cdots +\a_k)} \cdot 
    \prod_{i=1}^k \frac{1}{\G(\a_i)\G(-\a_i+1)} \cdot 
    \prod_{j=1}^{n+1}\frac{1}{\G(\a_{k+j} +1)} \cdot \bF(\a, \tx_0) .
  \end{align*}
  If $J=\{ 0,1,\dots ,k \}$, then the $J$-th entry of $\bS(\a;x_0)$ is $S(\a;x_0)$. 
  Otherwise, $J\in \cJ^{\circ}$ is expressed as 
  \begin{align*}
    J=\Big( \{ 0,1,\dots ,k \} -\{ i_1 ,\dots ,i_p \}\Big) 
      \cup \{ k+j_1 ,\dots ,k+j_p \} ,\\
    1\leq i_1< \cdots <i_p\leq k,\ 1\leq j_1< \cdots <j_p\leq n ,
  \end{align*}
  and the $J$-th entry of $\bS(\a;x_0)$ is 
  \begin{align*}
    \pm \dfrac{|\tx \la J \ra |}
    {\prod_{s=1}^p \a_{k+j_s}}
    \cdot \frac{\pa^p S(\a ;x_0)}{\pa x^0_{i_1 j_1}\cdots \pa x^0_{i_p j_p}} ,
  \end{align*}
  where we choose the pairs $(i_1,j_1),\dots ,(i_p,j_p)$ such that 
  none of them are $(\imath ,\jmath )$. 
  For the rule to determine this sign, refer to \cite[Section 4]{GM-PC}.
\end{lemma}
Note that $\displaystyle \pm \frac{\pa S(\a ;x_0)}{\pa x^0_{\imath \jmath}}$ does not appear 
because $J^{\van}\not\in \cJ^{\circ}$. 
\begin{exam}
  We consider the case when 
  $k=n=2$ and $x^0_{11}=0$ ($\imath =\jmath =1$, $J^{\van}=\{0,2,3\}$). 
  As in \cite[Example 7.10]{GM-PC}, if $x\in X^{(0)}$, we have
  \begin{align*}
    &S(\a; x) \\
    &=(\textrm{constant})\cdot \int_{\triangle} \prod_{j=1}^5 L_j^{\a_j} \cdot \tr \bigl( 
      \f \la 012 \ra , \f \la 013 \ra , \f \la 014 \ra ,  
      \f \la 023 \ra , \f \la 024 \ra , \f \la 034 \ra 
      \bigr) \\    
    &=\tr \Bigl(
      S(\a;x) ,
      \frac{x_{21}}{\a_3}\cdot \frac{\pa S(\a;x)}{\pa x_{21}}, 
      \frac{x_{22}}{\a_4}\cdot \frac{\pa S(\a;x)}{\pa x_{22}}, \\
    & \qquad \qquad  
      \frac{-x_{11}}{\a_3}\cdot \frac{\pa S(\a;x)}{\pa x_{11}},
      \frac{-x_{12}}{\a_4}\cdot \frac{\pa S(\a;x)}{\pa x_{12}}, 
      \frac{x_{11}x_{22}-x_{12}x_{21}}{\a_3 \a_4}\cdot \frac{\pa^2 S(\a;x)}{\pa x_{11} \pa x_{22}}
      \bigr) .
  \end{align*}
  When we consider $S(\a;x_0)$, the fourth entry, which corresponds to $\f \la J^{\van} \ra$,
  is deleted. 
  The limit of the sixth entry as $x_{11} \to 0$ cannot be taken directly. 
  However, it is known that 
  \begin{align*}
    \frac{\pa^2 S(\a;x)}{\pa x_{11} \pa x_{22}} =\frac{\pa^2 S(\a;x)}{\pa x_{12} \pa x_{21}}
  \end{align*}
  (see \cite[Section 4]{GM-PC}). Thus, we have 
  \begin{align*}
    S(\a; x_0) 
    &=\tr \Bigl(
      S(\a;x_0) ,
      \frac{x^0_{21}}{\a_3}\cdot \frac{\pa S(\a;x_0)}{\pa x^0_{21}}, 
      \frac{x^0_{22}}{\a_4}\cdot \frac{\pa S(\a;x_0)}{\pa x^0_{22}}, \\
    & \qquad \qquad  
      \frac{-x^0_{12}}{\a_4}\cdot \frac{\pa S(\a;x_0)}{\pa x^0_{12}}, 
      \frac{-x^0_{12}x^0_{21}}{\a_3 \a_4}\cdot \frac{\pa^2 S(\a;x_0)}{\pa x^0_{12} \pa x^0_{21}}
      \bigr) 
  \end{align*}
  by taking the limit $x_{11}\to 0$. 
\end{exam}
In general, we can retake the pairs $(i_1,j_1),\dots ,(i_p,j_p)$ in Lemma \ref{lem:bold-S} 
so that none of them are $(\imath ,\jmath )$. 

Applying the results in Subsection \ref{subsection:contiguity-integral}, 
we obtain the contiguity relations for $\bS(\a;x_0)$. 
\begin{cor}
  \label{cor:contiguity-S}
  We set $\conti{l}^0(\a;x_0)=\conti{l}(\a;\tx_0)$, as given in 
  Theorem \ref{th:contiguity} (recall that we set $j_0=0$ and $q=k+n+1$). 
  For $1\leq i \leq k$ and $1\leq j \leq n+1$, we have
  \begin{align*}
    \bS(\a^{(i)};x_0)&=\conti{i}^0(\a;x_0) \bS(\a;x_0), \\
    \bS(\a^{(k+j)};x_0)&=\frac{1}{\a_{k+j}+1}\conti{k+j}^0(\a;x_0) \bS(\a;x_0) .
  \end{align*}
\end{cor}
Similar to \cite[Section 7]{GM-PC}, 
we can utilize these contiguity relations to evaluate the normalizing constants 
of the two-way contingency tables 
with fixed marginal sums that have exactly one zero cell.

\subsection{Inverse of some intersection matrices}
To implement such evaluations, 
we need explicit expressions for the inverse of the intersection matrices
\begin{align*}
  &C=\Bigl( \cI^c_0 (\f \la I \ra , \f \la J \ra )\Bigr)_{I,J\in \cJ^{\circ}} , \\
  &P_l=\Bigl( \cI^c_0 (\f \la I \ra , \f \la J \ra ) 
    \Bigr)_{I\in \under{j_0}{\cJ^{\circ}}{l},J\in \cJ^{\circ}} ,
    Q_l=\Bigl( \cI^c_0 (\f \la I \ra , \f \la J \ra ) 
    \Bigr)_{I\in \under{l}{\cJ^{\circ}}{j_0},J\in \cJ^{\circ}} ,
\end{align*}
where $\cJ^{\circ} =\under{q}{\cJ}{j_0}-\{ J^{\van} \}$. 
We set 
\begin{align*}
  R_l = \Bigl( \cI^c_0 (\f \la I \ra , \f \la J \ra ) 
  \Bigr)_{I\in \under{j_0}{\cJ^{\circ}}{l},J\in \under{l}{\cJ^{\circ}}{j_0}} .
\end{align*}
By the same argument used in the proof of \cite[Proposition A.1]{GM-PC}, 
we have 
\begin{align*}
  &C^{-1}=R_q^{-1} \cdot 
    \Bigl( \cI^c_0 (\f \la I \ra , \f \la J \ra )\Bigr)_{I,J\in \under{j_0}{\cJ^{\circ}}{q}} 
    \cdot R_l^{-1}, \\
  &P_l^{-1}=R_q^{-1} \cdot 
    \Bigl( \cI^c_0 (\f \la I \ra , \f \la J \ra ) 
    \Bigr)_{I\in \under{j_0}{\cJ^{\circ}}{q},J\in \under{l}{\cJ^{\circ}}{j_0}} 
    \cdot R_l^{-1}, \\
  &Q_l^{-1}=R_q^{-1} \cdot 
    \Bigl( \cI^c_0 (\f \la I \ra , \f \la J \ra ) 
    \Bigr)_{I\in \under{j_0}{\cJ^{\circ}}{q},J\in \under{j_0}{\cJ^{\circ}}{l}} 
    \cdot R_l^{-1}.
\end{align*}
Thus, it suffices to give an explicit expression of $R_l^{-1}$. 
As in Subsection \ref{subsection:contiguity-integral}, 
we align the elements in $\under{l}{\cJ^{\circ}}{j_0}$ and $\under{j_0}{\cJ^{\circ}}{l}$
according to the correspondence 
\begin{align*}
  \under{l}{\cJ^{\circ}}{j_0} \ni 
  J\leftrightarrow \under{j_0}{J}{l}
  \in \under{j_0}{\cJ^{\circ}}{l} .
\end{align*}
\begin{proposition}
  Let $N_l=(n_{IJ})_{I\in \under{j_0}{\cJ^{\circ}}{l},J\in \under{l}{\cJ^{\circ}}{j_0}}$ 
  be a matrix defined as follows. 
  \begin{enumerate}[(i)]
  \item If $l\not\in J^{\van}$, then $N_l=O$. 
  \item If $l\in J^{\van}$, then
    \begin{align*}
      n_{IJ}=
      \begin{cases}
        \frac{\prod_{j\in J^{\van}-\{ j_0,l\}} \a_j}{\a_q} \cdot \a_{p_1} \a_{p_2} & 
        (I=\under{j_0}{(\under{l}{J^{\van}}{p_1})}{l},\ J=\under{l}{J^{\van}}{p_2}), \\
        0 & (\textrm{otherwise}) ,
      \end{cases}
    \end{align*}
    where $p_1,p_2\not\in J^{\van}\cup \{ q\}$.
  \end{enumerate}
  Then, $R_l^{-1}$ is expressed as
  \begin{align*}
    R_l^{-1} =\frac{1}{(\tpi)^k } \cdot \left(
    \diag \left( \dots, \prod_{j\in J-\{ j_0\}} \a_j ,\dots 
    \right)_{J\in \under{l}{\cJ^{\circ}}{j_0}} +N_l \right).
  \end{align*}
\end{proposition}
\begin{proof}
  By Corollary \ref{cor-intesection-cohomology}, 
  the intersection matrix $R_l$ is written as 
  \begin{align}
    \label{eq:expression-Rl-1}
    R_l = 
    &\diag\left( \dots, \cI^c(\f \la \under{j_0}{J}{l}\ra ,\f \la J\ra ),\dots 
      \right) \\
    \nonumber
    &-\frac{1}{\cI^c(\f^{\van}_{+},\f^{\van}_{-})} 
      \begin{pmatrix}
        \vdots \\ \cI^c(\f\la \under{j_0}{J}{l}\ra , \f^{\van}_{-}) \\ \vdots
      \end{pmatrix}
    \begin{pmatrix}
      \cdots & \cI^c(\f^{\van}_{+},\f \la J\ra) & \cdots
    \end{pmatrix} ,
    % \\
    % \nonumber
    % =&\diag\left( \dots,\frac{(\tpi)^k}{\prod_{j\in J-\{ 0\}} \a_j} ,\dots 
    %    \right) \\
    % \nonumber
    % &-\frac{\prod_{j\in J^{\van}}\a_j}{(2\pi\sqrt{-1})^k\sum_{j\in J^{\van}}\a_j}
    %   \begin{pmatrix}
    %     \vdots \\ \cI^c(\f\la \under{0}{J}{l}\ra , \f^{\van}_{-}) \\ \vdots
    %   \end{pmatrix}
    % \begin{pmatrix}
    %   \cdots & \cI^c(\f^{\van}_{+},\f \la J\ra) & \cdots 
    % \end{pmatrix} ,
  \end{align}
  where the index $J$ runs over $\under{l}{\cJ^{\circ}}{j_0}$. 
  First, we assume $l\not\in J^{\van}$. 
  We then have that $\cI^c(\f\la J \ra , \f^{\van}_{-}) =0$ 
  for any $J\in \under{j_0}{\cJ^{\circ}}{l}$ (see also the discussion
  following (\ref{eq:self-intersection-van-cohomology}); 
  $\under{j_0}{\cJ^{\circ}}{l}$ coincides with $\cJ^{\perp}$). 
  Hence, the column vector in (\ref{eq:expression-Rl-1}) is zero, and 
  $R_l^{-1}$ is easily obtained from 
  \begin{align*}
    \cI^c(\f \la \under{j_0}{J}{l}\ra ,\f \la J\ra )
    =\frac{(\tpi)^k}{\prod_{j\in J-\{ j_0\}} \a_j} .
  \end{align*}
  Next, we assume that $l\in J^{\van}$. 
  By a straightforward calculation, we can write
  the inverse of the matrix in (\ref{eq:expression-Rl-1}) as 
  \begin{align}
    \label{eq:expression-Rl-2}
    R_l^{-1}
    =&\diag\left( \dots,\frac{1}{\cI^c(\f \la \under{j_0}{J}{l}\ra ,\f \la J\ra )} ,\dots 
       \right) \\
    \nonumber
     &+\frac{1}{c_2}
       \begin{pmatrix}
         \vdots \\ \frac{\cI^c(\f\la \under{j_0}{J}{l}\ra , \f^{\van}_{-})}{
           \cI^c(\f \la \under{j_0}{J}{l}\ra ,\f \la J\ra )} \\ \vdots
       \end{pmatrix}
    \begin{pmatrix}
      \cdots & \frac{\cI^c(\f^{\van}_{+},\f \la J\ra)}{
          \cI^c(\f \la \under{j_0}{J}{l}\ra ,\f \la J\ra )} & \cdots
    \end{pmatrix} ,
  \end{align}
  where 
  \begin{align*}
    c_2 = \cI^c(\f^{\van}_{+},\f^{\van}_{-})
    -\sum_{J\in \under{l}{\cJ^{\circ}}{j_0}}
    \frac{\cI^c(\f\la \under{j_0}{J}{l}\ra , \f^{\van}_{-}) \cdot \cI^c(\f^{\van}_{+},\f \la J\ra)}{
    \cI^c(\f \la \under{j_0}{J}{l}\ra ,\f \la J\ra )}.
  \end{align*}
  As this is nothing but $c_2$ in the proof of Proposition \ref{prop:cohomology-basis-Z1-2} 
  with $(l,l',p)$ replaced by $(l,j_0,q)$, 
  we have 
  % For $J\in \under{l}{\cJ^{\circ}}{0}$, we have
  % $\cI^c(\f\la \under{0}{J}{l}\ra , \f^{\van}_{-})\neq 0$, $\cI^c(\f^{\van}_{+},\f \la J\ra) \neq 0$ 
  % if and only if 
  % $J=\under{l}{J^{\van}}{p}$ with $p\not\in J^{\van}\cup \{ q\}$. 
  % Thus, 
  \begin{align*}
    % n_0 
    % &=\sum_{p\not\in J^{\van}\cup \{ q\}} 
    %   \frac{-\frac{(\tpi)^k}{\prod_{j\in J^{\van}-\{ 0\}} \a_j} \cdot 
    %   \frac{(\tpi)^k}{\prod_{j\in J^{\van}-\{ l\}} \a_j}}{
    %   \frac{(\tpi)^k}{\prod_{j\in (J-\{ 0 ,l\})\cup \{ p\}} \a_j}}
    %   -\frac{(2\pi\sqrt{-1})^k\sum_{j\in J^{\van}}\a_j}{\prod_{j\in J^{\van}}\a_j} \\
    % &=\sum_{p\not\in J^{\van}\cup \{ q\}} 
    %   \frac{-(\tpi)^k \a_p }{\prod_{j\in J^{\van}}\a_j}
    %   -\frac{(2\pi\sqrt{-1})^k\sum_{j\in J^{\van}}\a_j}{\prod_{j\in J^{\van}}\a_j} \\
    % &=\frac{-(\tpi)^k \sum_{j\neq q} \a_j}{\prod_{j\in J^{\van}}\a_j}
    c_2 =-\frac{(\tpi)^k \a_q}{\prod_{j\in J^{\van}}\a_j} .
  \end{align*}
  We consider the second term on the right-hand side of (\ref{eq:expression-Rl-2}).
  For $p_1,p_2\not\in J^{\van}\cup \{ q\}$, 
  the $(\under{j_0}{(\under{l}{J^{\van}}{p_1})}{l},\under{l}{J^{\van}}{p_2})$-entry is 
  \begin{align*}
    &\frac{1}{c_2} \cdot  
      \frac{\cI^c(\f\la \under{j_0}{(\under{l}{J^{\van}}{p_1})}{l} \ra , \f^{\van}_{-})\ }{
      \frac{(\tpi)^k}{\prod_{j\in (J^{\van}-\{ j_0 ,l\})\cup \{ p_1\}} \a_j}} \cdot 
      \frac{\ \cI^c(\f^{\van}_{+},\f \la \under{l}{J^{\van}}{p_2}\ra)\ }{
      \frac{(\tpi)^k}{\prod_{j\in (J^{\van}-\{ j_0 ,l\})\cup \{ p_2\}} \a_j}} \\
    &=-\frac{\prod_{j\in J^{\van}}\a_j}{(\tpi)^k \a_q} 
      \cdot \frac{-\a_{p_1}}{\a_l} \cdot \frac{\a_{p_2}}{\a_{j_0}}
      =\frac{1}{(\tpi)^k}\cdot 
      \frac{\prod_{j\in J^{\van}-\{ j_0 ,l\}} \a_j}{\a_q} \cdot \a_{p_1} \a_{p_2}, 
  \end{align*}
  and the others are zero. 
  Thus, we obtain the proposition. 
\end{proof}

\end{document}